\documentclass[letterpaper, 11pt,  reqno]{amsart}

\usepackage{amsmath,amssymb,amscd,amsthm,amsxtra, esint}
\usepackage{tikz} 

\usepackage{comment}
\usepackage{mathtools}
\usepackage{color}
\usepackage[implicit=true]{hyperref}

\usepackage{cases}

\usepackage[left=32mm, right=32mm, 
bottom=27mm]{geometry}

\allowdisplaybreaks[2]

\usepackage{color}

\definecolor{gr}{rgb}   {0.,   0.69,   0.23 }
\definecolor{bl}{rgb}   {0.,   0.5,   1. }
\definecolor{mg}{rgb}   {0.85,  0.,    0.85}
\definecolor{yl}{rgb}   {0.8,  0.7,   0.}
\definecolor{or}{rgb}  {0.7,0.2,0.2}

\usetikzlibrary{shapes.misc}
\usetikzlibrary{shapes.symbols}
\usetikzlibrary{shapes.geometric}

\tikzset{
	ddot/.style={circle,fill=white,draw=black,inner sep=0pt,minimum size=0.8mm},
	>=stealth,
	}

\tikzset{
	ddot2/.style={circle,fill=black,draw=black,inner sep=0pt,minimum size=0.8mm},
	>=stealth,
	}

\newtheorem{theorem}{Theorem} [section]

\newtheorem{lemma}[theorem]{Lemma}
\newtheorem{proposition}[theorem]{Proposition}
\newtheorem{remark}[theorem]{Remark}

\newtheorem{corollary}[theorem]{Corollary}

\newcommand{\noi}{\noindent}
\newcommand{\Z}{\mathbb{Z}}
\newcommand{\R}{\mathbb{R}}

\newcommand{\T}{\mathbb{T}}

\let\P= \undefined
\newcommand{\P}{\mathbf{P}}

\newcommand{\E}{\mathbb{E}}

\renewcommand{\L}{\mathcal{L}}

\newcommand{\F}{\mathcal{F}}

\newcommand{\al}{\alpha}

\newcommand{\dl}{\delta}

\newcommand{\nb}{\nabla}

\newcommand{\Dl}{\Delta}
\newcommand{\eps}{\varepsilon}

\newcommand{\G}{\Gamma}
\newcommand{\ld}{\lambda}

\newcommand{\s}{\sigma}

\newcommand{\ft}{\widehat}

\newcommand{\cj}{\overline}

\newcommand{\dt}{\partial_t}

\newcommand{\ta}{\theta}

\renewcommand{\l}{\ell}
\renewcommand{\o}{\omega}
\renewcommand{\O}{\Omega}

\newcommand{\les}{\lesssim}

\newcommand{\jb}[1]
{\langle #1 \rangle}

\newcommand{\jbb}[1]
{[\hspace{-0.6mm}[ #1 ]\hspace{-0.6mm}]}

\newcommand{\ind}{\mathbf 1}

\newcommand{\N}{\mathbb{N}}

\renewcommand{\H}{\mathcal{H}}

\newtheorem*{ackno}{Acknowledgements}

\newcommand{\too}{\longrightarrow}

\newcommand{\rhoo}{\vec{\rho}}
\newcommand{\muu}{\vec{\mu}}
\newcommand{\deff}{\stackrel{\textup{def}}{=}}

\newcommand{\Id}{\textup{Id}}

\newcommand{\wick}[1]{:\!{#1}\!:}

\numberwithin{equation}{section}
\numberwithin{theorem}{section}

\makeatletter
\@namedef{subjclassname@2020}{%
  \textup{2020} Mathematics Subject Classification}
\makeatother

\begin{document}
\baselineskip = 14pt

\title[On the 2-$d$ stochastic viscous NLW]
{On the two-dimensional singular stochastic viscous nonlinear wave equations}

\author[R.~Liu and  T.~Oh]
{Ruoyuan Liu and  Tadahiro Oh}

\address{
Ruoyuan Liu,  School of Mathematics\\
The University of Edinburgh\\
and The Maxwell Institute for the Mathematical Sciences\\
James Clerk Maxwell Building\\
The King's Buildings\\
Peter Guthrie Tait Road\\
Edinburgh\\ 
EH9 3FD\\
 United Kingdom}

\email{ruoyuan.liu@ed.ac.uk}


\address{
Tadahiro Oh, School of Mathematics\\
The University of Edinburgh\\
and The Maxwell Institute for the Mathematical Sciences\\
James Clerk Maxwell Building\\
The King's Buildings\\
Peter Guthrie Tait Road\\
Edinburgh\\ 
EH9 3FD\\
 United Kingdom}

\email{hiro.oh@ed.ac.uk}

%


\begin{abstract}
We 
study the stochastic 
viscous nonlinear wave equations (SvNLW) on $\T^2$, 
forced by a fractional derivative of the space-time white noise $\xi$.
In particular, we consider SvNLW with the singular additive forcing $D^\frac{1}{2}\xi$
such that solutions are expected to be merely distributions.
By introducing an appropriate renormalization, 
we prove local well-posedness of SvNLW.
By establishing an energy bound via a Yudovich-type argument, 
we also prove pathwise global well-posedness of the defocusing cubic SvNLW.
Lastly, 
in  the defocusing case, we prove almost sure global well-posedness
of SvNLW with respect to certain  Gaussian random initial data.

\end{abstract}

%
%
%
%
%
%

\maketitle

%
%


\tableofcontents

\section{Introduction}
\label{SEC:1}

\subsection{Stochastic viscous nonlinear wave equations}

In  \cite{KC1}, 
Kuan and \v{C}ani\'c 
proposed the following wave equation on $\R^2$
augmented by the viscous effect:
\begin{align}
\dt^2 u - \Dl  u  + 2\mu D \dt u  = F_\text{ext}(u), 
\label{NLW1}
\end{align}

\noi
where $\mu > 0$ is a constant,  $D = |\nb| = \sqrt{-\Dl}$, 
and 
$F_\text{ext}(u)$ 
denotes an external forcing, which may in general depend on the unknown $u$.
The equation \eqref{NLW1}
appears
in the study of 
 fluid-structure interaction 
 in the three-dimensional space where the Dirichlet-Neumann operator models the coupling between a viscous, incompressible fluid and an elastic structure.
Here, 
 the viscosity term $2\mu D \dt u$
in~\eqref{NLW1}
represents the effect of the Cauchy stress tensor
of Newtonian fluid
in the vertical direction (namely, in $z$-direction).
See \cite{KC1} for the derivation of \eqref{NLW1}.

The general solution to  the homogeneous linear viscous wave equation:
\begin{align}
\dt^2 u - \Dl  u  + 2\mu D \dt u  = 0
\label{NLW}
\end{align}

\noi
is given by 
\[
 u(t) =  e^{(-\mu |\nb|  + \sqrt{(\mu^2 - 1) |\nb|^2})t }
f_1 + e^{(-\mu |\nb|  - \sqrt{(\mu^2 - 1) |\nb|^2})t } f_2.
\]

\noi
When $\mu\geq 1$, we have 
$-\mu |\xi|  + \sqrt{(\mu^2 - 1) |\xi|^2}\sim - \mu^{-1} |\xi|$
and thus the equation \eqref{NLW1} is purely of parabolic type.
In this case, we can study well-posedness
of \eqref{NLW1}, 
simply by using the Schauder estimate
for the Poisson kernel (see Lemma \ref{LEM:Sch} below).
On the other hand, when $0 < \mu < 1$, 
the solution to \eqref{NLW} with initial data $(u, \dt u) |_{t = 0} = (u_0, u_1)$
is given by 
\begin{align}
\begin{split}
u & = 
 e^{-\mu D t} \bigg(\cos (\sqrt{1 - \mu^2} D t ) 
+ \frac{\mu }{\sqrt{1 - \mu^2} }\sin (\sqrt{1 - \mu^2} D t )  \bigg)u_0\\
& \quad + e^{-\mu D t} \, \frac{\sin (\sqrt{1 - \mu^2} D t ) }{\sqrt{1- \mu^2} D} u_1, 
\end{split}
\label{NLW2}
\end{align}

\noi
and thus 
we see that 
the equation exhibits an interesting mixture
of the wave dispersion and the parabolic regularization 
by the fluid viscosity.
For this reason, we will restrict our attention to $0 < \mu < 1$.
Note that, when $0 < \mu < 1$, 
 a precise value of $\mu$ (and $\sqrt{1 - \mu^2 }$)
 in~\eqref{NLW2}
does not play any important role in terms of the well-posedness theory, 
and thus 
without loss of generality, we set $\mu = \frac 12$ 
in the remaining part of the paper.
See also Footnote \ref{FT:mu} below.

In \cite{KC1}, 
Kuan and \v{C}ani\'c studied 
well-posedness and ill-posedness of the following viscous nonlinear wave equation (vNLW)
on $\R^2$:
\begin{align*}
\dt^2 u - \Dl  u  + D \dt u  + u^k  = 0
\end{align*}

\noi
in both the deterministic and probabilistic settings
(in particular with random initial data).
See also \cite{CKO, Liu}.
In a recent preprint \cite{KC2}, 
Kuan and \v{C}ani\'c also studied the following stochastic viscous wave equation
with a multiplicative noise
on $\R^d$, $d = 1, 2$:
\begin{align*}
\dt^2 u - \Dl  u  + D \dt u  = f(u) \xi, 
\end{align*}

\noi
where $f$ is a Lipschitz function and $\xi$ denotes the (Gaussian) space-time white noise on $\R_+\times \R^2$.


In this paper, we consider   the  following stochastic vNLW (SvNLW)
with an additive stochastic forcing
on the two-dimensional torus $\T^2 = (\R/\Z)^2$:
\begin{align}
\dt^2 u -\Dl u + D \dt u + u^k
=  D^\al \xi
\label{SNLW1}
\end{align}

\noi
where $\al \geq  0$ and 
 $\xi$ denotes the (Gaussian) space-time white noise on $\R_+\times \T^2$.
By a standard argument (see, for example, Lemma \ref{LEM:Psi} below), 
we see
 that the stochastic convolution $\Psi$,  satisfying 
\begin{align*}
\dt^2 \Psi -\Dl \Psi + D \dt \Psi
=  D^\al \xi
\end{align*}

\noi
(say, with the zero initial data), 
is almost surely a continuous function on 
$\R_+\times \T^2$,  when $\al < \frac 12$.
It is worthwhile to note that 
a combination of the wave dispersion and the dissipation
by the fluid viscosity
yields $\frac 32$-smoothing on the noise (rather than 
the usual one degree of smoothing
for stochastic heat equations \cite{DPD, MW2}
and stochastic wave equations~\cite{GKO, OOk}).
For this reason, 
 we set  $\al = \frac 12$  in this paper and study 
the following Cauchy problem for SvNLW on $\T^2$:
\begin{align}
\begin{cases}
\dt^2 u + (1-\Dl) u + D \dt u + u^k
=  \sqrt 2 D^\frac{1}{2} \xi\\
(u,\dt u)|_{t = 0}=(u_0,u_1).
\end{cases}
\label{SNLW3}
\end{align}

\noi
In this case, the corresponding stochastic convolution 
is merely a distribution and thus we need to introduce a proper renormalization
to give a precise meaning to the equation.

\begin{remark}\label{REM:Gibbs}\rm
In \eqref{SNLW3}, 
we replaced $-\Dl$ 
by $1-\Dl$. This modifications simplifies
part of the argument (so that we do not need to make a separate analysis
at the zeroth frequency). 
Furthermore, 
this modification, along with the extra factor $\sqrt 2$, 
is necessary for the almost sure global well-posedness result
(Theorem \ref{THM:GWP2}).
Note that Theorems \ref{THM:LWP} and \ref{THM:GWP1}
apply to~\eqref{SNLW1} with $\al = \frac 12$
with essentially identical proofs.

\end{remark}

\subsection{Renormalized SvNLW}\label{SUBSEC:1.2}
In this subsection, we briefly go over the renormalization procedure for \eqref{SNLW3}, 
following the discussion in \cite{GKO, OOk, GKOT}.
Let $\Psi$ be the solution to the following linear stochastic viscous wave equation:
\begin{align}
\begin{cases}
\dt^2 \Psi + (1-\Dl) \Psi + D \dt \Psi
= \sqrt{2} D^\frac{1}{2} \xi\\
(\Psi,\dt \Psi)|_{t = 0}=(0,0).
\end{cases}
\label{SNLW4}
\end{align}

\noi
By writing in the Duhamel formulation (= mild formulation), 
the stochastic convolution $\Psi$ can be expressed as
\begin{equation}
\Psi (t) =  \sqrt{2} \int_0^t S(t-t') D^\frac{1}{2} dW(t'),
\label{psi1}
\end{equation}

\noi
where  the linear propagator $S(t)$ is defined by
\begin{equation}
S(t) = e^{-\frac{D}{2} t} \frac{\sin (t \jbb{D} )}{\jbb{D}}\qquad 
\text{with }
\jbb{D} = \sqrt{1-\tfrac{3}{4}\Dl}
\label{S1}
\end{equation}

\noi
and 
  $W$ denotes a cylindrical Wiener process on $L^2(\T^2)$:\footnote{Hereafter, we drop the harmless factor $2\pi$.} 
\begin{align*}
W(t)
 = \sum_{n \in \Z^2 } B_n (t) e_n.
\end{align*}

\noi
Here, $e_n(x) = e^{2\pi i n \cdot x}$
and 
$\{ B_n \}_{n \in \Z^2}$ 
is defined by 
$B_n(t) = \jb{\xi, \ind_{[0, t]} \cdot e_n}_{t, x}$, 
where  $\jb{\cdot, \cdot}_{t, x}$ denotes 
the duality pairing on $\R_+ \times \T^2$.
As a result, 
we see that $\{ B_n \}_{n \in \Z^2}$ is a family of mutually independent complex-valued\footnote
{In particular, $B_0$ is  a standard real-valued Brownian motion.
Note that we have, for any $n \in \Z^2$,  
 \[\text{Var}(B_n(t)) = \E\big[
 \jb{\xi, \ind_{[0, t]} \cdot e_n}_{t, x}\cj{\jb{\xi, \ind_{[0, t]} \cdot e_n}_{t, x}}
 \big] = \|\ind_{[0, t]} \cdot e_n\|_{L^2_{t, x}}^2 = t.\]
} 
Brownian motions conditioned so that $B_{-n} = \cj{B_n}$, $n \in \Z^2$.

Given $N \in \N$, we define the truncated stochastic convolution $\Psi_N = \P_N \Psi$, where $\P_N$ denotes the frequency cutoff onto the spatial frequencies $\{ |n| \leq N \}$. Then, 
for each fixed $t \geq 0$ and $x \in \T^2$, 
a direct computation shows that
 $\Psi_N (t,x)$ is a mean-zero real-valued Gaussian random variable with variance
\begin{align}
\begin{split}
\s_N (t)
& \deff  \E \big[\Psi_N (t,x)^2\big]
= 2 \sum_{\substack{n \in \Z^2\\ |n| \leq N}}
\int_0^t 
e^{- (t - t')|n|} \bigg[\frac{\sin ((t-t')\jbb{n} )}{\jbb{n}} \bigg]^2 |n| dt'\\
& =  \sum_{\substack{n \in \Z^2\\ |n| \leq N}}
\frac{1}{\jb{n}^2} -   \frac{e^{-t|n|}}{\jbb{n}^2}
\bigg( 1 -   \frac{ |n|^2}{4\jb{n}^2} \cos (2\jbb{n}t)
+    \frac{ |n|\jbb{n}}{2\jb{n}^2} \sin (2\jbb{n}t)\bigg)
\\
&  \sim \sum_{\substack{n \in \Z^2\\ |n| \leq N}} \frac{1}{\jb{n}^{2}} \sim \log N
\too \infty, 
\end{split}
\label{sig1}
\end{align}

\noi
as $N \to \infty$, where 
\[ \jb{n} = ( 1 + |n|^2)^\frac{1}{2}\qquad
\text{and} \qquad \jbb{n} = \sqrt{1+\tfrac{3}{4}|n|^2}\, .\]

\noi
From this computation, we see  that $\{\Psi_N(t)\}_{N \in \N}$
is almost surely unbounded in $W^{0, p}(\T^2)$ for any $1 \leq p \leq \infty$.

Let us now consider   the truncated SvNLW with the regularized noise:
\begin{align}
\begin{cases}
\dt^2 u_N + (1-\Dl) u_N + D \dt u_N + u_N^k
= \sqrt{2} D^\frac{1}{2} \P_N \xi\\
(u_N,\dt u_N)|_{t = 0}=(u_0,u_1).
\end{cases}
\label{SNLW5}
\end{align}

\noi
Proceeding with 
the first order expansion (\cite{McKean, BO96, DPD}):
\begin{equation}
u_N = \Psi_N + v_N, 
\label{decomp1}
\end{equation}

\noi
we see that the residual term $v_N$ satisfies the following equation:
\begin{equation}
\dt^2 v_N + (1-\Dl) v_N + D \dt v_N + \sum_{\ell = 0}^k \binom{k}{\ell} \Psi_N^\ell v_N^{k-\ell} = 0.
\label{SNLW6}
\end{equation}

\noi
Note that the power $\Psi_N^\ell$ does not converge to any limit as $N \to \infty$. 
This is where we introduce the Wick renormalization:
\begin{equation}
\wick{\Psi_N^\ell (t,x)} \, \deff H_\ell (\Psi_N(t,x); \s_N(t)),
\label{Wick1}
\end{equation}

\noi
where $H_\ell (x,\s)$ is the Hermite polynomial of degree $\ell$ with variance parameter $\s$. See Subsection \ref{SUBSEC:sto}. 
This yields the renormalized version of \eqref{SNLW6}:
\begin{equation}
\dt^2 v_N + (1-\Dl) v_N + D \dt v_N + \sum_{\ell = 0}^k \binom{k}{\ell} \wick{\Psi_N^\ell} v_N^{k-\ell} = 0.
\label{SNLW6a}
\end{equation}

\noi
In Lemma \ref{LEM:Psi}, we show that the Wick power $\wick{\Psi_N^\ell}$ converges to a limit $\wick{\Psi^\ell}$ in $C([0,T];W^{-\eps,\infty} (\T^2))$ for any $\eps > 0$ and $T > 0$, almost surely.
Then, 
by taking $N \to \infty$, we obtain the limiting equation:
\begin{equation}
\dt^2 v + (1-\Dl) v + D \dt v + \sum_{\ell = 0}^k \binom{k}{\ell} \wick{\Psi^\ell} v^{k-\ell} = 0.
\label{SNLW7}
\end{equation}

At the level of $u_N$, in view of \eqref{decomp1}, 
we define the renormalized nonlinearity   $\wick{u_N^k}$ by 
\begin{align}
\wick{u_N^k} \, = \, \wick{(\Psi_N + v_N)^k} \, = \sum_{\ell = 0}^k \binom{k}{\ell} \wick{\Psi_N^\ell} v_N^{k-\ell}.
\label{Wick2}
\end{align}

\noi
Then, if $v_N$ solves \eqref{SNLW6a}, 
then $u_N = \Psi_N + v_N$ satisfies  the following truncated renormalized SvNLW:
\begin{align}
\dt^2 u_N + (1-\Dl) u_N + D \dt u_N \, + \wick{u_N^k}\, 
= \sqrt{2} D^\frac{1}{2} \P_N \xi.
\label{SNLW8}
\end{align}

\noi
Similarly, if  $v$ solves \eqref{SNLW7}, 
then $u = \Psi + v$ satisfies  the following  renormalized SvNLW:
\begin{align}
\dt^2 u + (1-\Dl) u + D \dt u \, \, +  \wick{u^k}\, 
= \sqrt{2} D^\frac{1}{2}  \xi, 
\label{SNLW9}
\end{align}

\noi
where the renormalized nonlinearity 
$\wick{u^k}$  is defined as in \eqref{Wick2} (by dropping the subscript~$N$).

\subsection{Main results}

Our main goal is to study well-posedness of the renormalized SvNLW \eqref{SNLW9}.
More precisely, we study the following Duhamel 
formulation of \eqref{SNLW7} endowed with initial data 
$(v, \dt v)|_{t = 0} = (u_0, u_1)$:
\begin{equation}
v(t) = V(t)(u_0, u_1) - \sum_{\ell = 0}^k \binom{k}{\ell}
\int_0^t S(t - t') 
 \wick{\Psi^\l} v^{k-\l} (t') dt', 
\label{SNLW10}
\end{equation}

\noi
where the linear propagator $V(t)$  is defined by 
\begin{equation}
V(t) (u_0, u_1)
= e^{-\frac{D}{2} t} \bigg(\cos (t\jbb{D} ) 
+ \frac{D}{2 \jbb{D}}\sin (t \jbb{D} ) \bigg)u_0
+ e^{-\frac{D}{2} t} \frac{\sin (t\jbb{D} )}{\jbb{D}} u_1.
\label{S2}
\end{equation}

\noi
Then, given the almost sure regularity of the Wick powers
$ \wick{\Psi^\l}$, 
standard deterministic analysis  
yields the following local well-posedness result.

\begin{theorem}\label{THM:LWP}
Let $k\geq 2$ be  an integer and $s \geq 1$.   Then, the renormalized SvNLW \eqref{SNLW9}
is locally well-posed in $\H^s(\T^2) = H^s(\T^2) \times H^{s-1}(\T^2)$
in the  sense that the following statement holds true almost surely; 
given $(u_0, u_1) \in \H^s(\T^2)$, 
there exists a unique local-in-time solution $v$ to \eqref{SNLW7}
with initial data $(v, \dt v)|_{t = 0} = (u_0, u_1)$, 
where $\{\,\wick{\Psi^\ell}\,\}_{\l = 1}^k$
denotes the stochastic convolution $\Psi$ defined in~\eqref{psi1}
and its Wick powers 
$\wick{\Psi^\ell}$, $\l = 2, \dots, k$,  
defined in Lemma~\ref{LEM:Psi} below.

Furthermore, 
there exists
 an almost surely positive stopping time
 $T = T(\o)$ such that 
the solution $u_N = \Psi_N + v_N$ to the truncated renormalized SvNLW \eqref{SNLW8} 
\textup{(}and its time derivative $\dt u_N$\textup{)} converges
to the solution $u = \Psi + v$ to \eqref{SNLW9} \textup{(}constructed above\textup{)} in $C([0, T]; H^{-\eps}(\T^2))$ 
\textup{(}and to $\dt u$ in $C([0, T]; H^{-1-\eps}(\T^2)), respectively$\textup{)}, $\eps > 0$, almost surely, 
as $N \to \infty$.
Here, 
$v_N$ denotes the solution to \eqref{SNLW6a}  
with $(v_N, \dt v_N)|_{t = 0} = (u_0, u_1)$.

\end{theorem}

See Proposition \ref{PROP:LWP2} below 
for the local well-posedness statement at the level of the 
residual term $v = u - \Psi$, satisfying  \eqref{SNLW10}.
We point out that  the regularity of initial data can be lowered
but we do not pursue this issue here.  See Remark \ref{REM:LWP3}.

\medskip

Next, we turn our attention to the global well-posedness problem.
In the cubic case ($k = 3$), 
we have the following pathwise global well-posedness result. 

\begin{theorem}\label{THM:GWP1}
Let $k = 3$  and $s \geq 1$.   Then, the renormalized cubic SvNLW \eqref{SNLW9}
is globally well-posed in $\H^s(\T^2)$
in the  sense that the following statement holds true almost surely; 
given $(u_0, u_1) \in \H^s(\T^2)$, 
there exists a unique global-in-time solution $v$ to \eqref{SNLW7}
with initial data $(v, \dt v)|_{t = 0} = (u_0, u_1)$, 
where $\Psi$, $\wick{\Psi^2}$, and $\wick{\Psi^3}$ 
are the stochastic convolution $\Psi$ defined in~\eqref{psi1}
and its Wick powers
defined in Lemma \ref{LEM:Psi} below.

\end{theorem}

In proving  Theorem \ref{THM:GWP1}, 
we study 
\eqref{SNLW7} with $k = 3$:
 \begin{equation}
\dt^2 v  + (1 - \Dl) v +  
D \dt v + 
v^3 + 3 v^2 \Psi +  3v :\!\Psi^2\!: + :\!\Psi^3\!: \, = 0.
\label{SNLW11}
\end{equation}

\noi
From the proof of Theorem \ref{THM:LWP}, 
we see that it suffices to control the $\H^1$-norm of $\vec v(t)  \deff  (v(t), \dt v(t))$.
For this purpose, we study the evolution of the energy (with $k = 3$)
\begin{equation}
E(\vec v) = \frac{1}{2} \int_{\T^2} \big( v^2 + |\nb v|^2 \big) dx + \frac{1}{2} \int_{\T^2} (\dt v)^2 dx + \frac{1}{k+1} \int_{\T^2} v^{k+1} dx
\label{E1}
\end{equation}

\noi
for the standard  nonlinear wave equation (NLW):
\begin{align}
\dt^2 u + (1-\Dl) u  + u^k = 0.
\label{NLW5}
\end{align}

\noi
As in the case of the stochastic NLW studied in \cite{GKOT}, 
the energy $E(\vec v)$ is not conserved under~\eqref{SNLW11}
due to the singular perturbative term
$3 v^2 \Psi +  3v :\!\Psi^2\!: + :\!\Psi^3\!: \, $.
For our problem, the dissipation by the viscous term 
comes in rescue and allows us to establish 
a double exponential growth bound on $E(\vec v)$ via 
a Yudovich-type argument \cite{Yud, BG}.
See
Section~\ref{SEC:GWP1} for details.
In \cite{BT3}, Burq and Tzvetkov 
 used an analogous 
Yudovich-type argument 
and proved 
probabilistic global well-posedness of
 the defocusing cubic
NLW,  \eqref{NLW5} with $k = 3$,  on the three-dimensional torus $\T^3$
with randomized initial data in $L^2(\T^3)$.
A key difference between Theorem \ref{THM:GWP1}
and \cite{BT3} is that, thanks to the dissipative smoothing effect, 
we can handle data (namely, the stochastic convolution and its Wick powers) of 
slightly negative regularity.

\medskip

Lastly, we consider global well-posedness of \eqref{SNLW9} with random initial data.
More precisely, consider a pair $(u_0^\o, u_1^\o)$ of random functions defined by 
\begin{equation}
u_0^\o = \sum_{n \in \Z^2} \frac{g_n(\o)}{\jb{n}}e_n
\qquad\text{and}\qquad
u_1^\o = \sum_{n \in \Z^2} h_n(\o)e_n.
\label{IV1}
\end{equation}

\noi
Here, $\{g_n,h_n\}_{n\in\Z^2}$ is a family of independent standard complex-valued Gaussian random variables such that $\cj{g_n}=g_{-n}$ and $\cj{h_n}=h_{-n}$, 
 $n \in \Z^2$.
We assume that 
$\{g_n,h_n\}_{n\in\Z^2}$ is independent from the space-time white noise $\xi$ in \eqref{SNLW9}.
A standard computation shows that 
$(u_0^\o, u_1^\o) \in \H^s(\T^2) \setminus \H^0(\T^2)$
for any $s < 0$, almost surely.  In particular, 
$(u_0^\o, u_1^\o)$ in~\eqref{IV1} is much rougher
than the $\H^1$-initial data considered in Theorems \ref{THM:LWP} and \ref{THM:GWP1}.
Our goal is to prove almost sure global well-posedness of \eqref{SNLW9}
 with respect to the Gaussian random initial data 
$(u_0^\o, u_1^\o)$ in \eqref{IV1}.

For this purpose,  let us first define
the stochastic convolution $\Phi$
with the Gaussian random initial data $(u_0^\o,u_1^\o)$ in~\eqref{IV1}:
\begin{align}
\begin{cases}
\dt^2 \Phi + (1-\Dl) \Phi + D \dt \Phi
= \sqrt{2} D^\frac{1}{2} \xi\\
(\Phi,\dt \Phi)|_{t = 0}=(u_0^\o,u_1^\o).
\end{cases}
\label{phi1}
\end{align}

\noi
By writing \eqref{phi1} in the Duhamel formulation, we have
\begin{align} 
\Phi (t) 
 = V(t)(u_0^\o,u_1^\o)
+   \sqrt{2} \int_0^t S(t-t') D^\frac{1}{2} dW(t'),
\label{phi2}
\end{align}   

\noi
where $V(t)$ is as in \eqref{S2}.
A direct computation with \eqref{sig1} and \eqref{IV1} shows that $\Phi_N(t, x)=\P_N\Phi( t, x)$
 is a mean-zero real-valued Gaussian random variable with variance
\begin{align}
\begin{split}
\al_N & \deff  \E \big[\Phi_N(t, x)^2\big] = 
 \E\big[\big(\P_NV(t)(u_0^\o,u_1^\o)(x) \big)^2\big]
 + 
 \E\big[\big(\P_N\Psi(t, x)\big)^2\big]\\
&  =\sum_{\substack{n\in\Z^2\\|n|\leq N}}\frac{1}{\jb{n}^2}
\sim \log N 
\end{split}
 \label{sig2}
\end{align}

\noi
for any $t\ge 0$, $x\in\T^2$,  and $N \ge 1$.
Note  that,  unlike $\s_N(t)$ in \eqref{sig1}, 
the variance $\al_N$ is time independent.
This is due to the fact that 
the distribution of $( \Phi_N(t), \dt  \Phi_N(t))$
is invariant under the linear dynamics \eqref{phi1}.
As in \eqref{Wick1}, we then define 
 the Wick power by 
\begin{align}
:\!\Phi_N^\l ( t, x) \!:
 \, \deff  H_{\l} (\Phi_N( t, x); \al_N)
\label{Wick3}
\end{align}

\noi
for $k \in \N$.
As before, it follows that the Wick power $\wick{\Phi_N^\ell}$ converges to a limit $\wick{\Phi^\ell}$ in $C([0,T];W^{-\eps,\infty} (\T^2))$ for any $\eps > 0$ and $T > 0$, almost surely;
see Lemma \ref{LEM:Psi}.
Then, by proceeding as in Subsection 
\ref{SUBSEC:1.2}, 
namely, by (i) first considering the truncated equation~\eqref{SNLW5}
with the random initial data $(u_0^\o, u_1^\o)$ in \eqref{IV1}, 
(ii) using 
the first order expansion 
$u_N = \Phi_N + v_N$
and introducing Wick renormalizations,
and (iii) taking a limit $N \to \infty$, 
we arrive at the following (renormalized) reformulation of \eqref{SNLW9}
in this setting:
\begin{equation}
\begin{cases}
\dt^2 v + (1-\Dl) v + D \dt v + \sum_{\ell = 0}^k \binom{k}{\ell} \wick{\Phi^\ell} v^{k-\ell} = 0\\
(v, \dt v)|_{t = 0} = (0, 0).
\end{cases}
\label{SNLW7a}
\end{equation}

We now state an almost sure global well-posedness result
of \eqref{SNLW9} with the random initial data  $(u_0^\o, u_1^\o)$ in \eqref{IV1}.

\begin{theorem}\label{THM:GWP2}
Let $k \in 2 \N + 1$.   Then, the renormalized SvNLW \eqref{SNLW9}
is almost surely globally well-posed 
with the random initial data 
 $(u_0^\o, u_1^\o)$ defined in \eqref{IV1}
in the  sense that the following statement holds true almost surely; 
there exists a unique global-in-time solution $v$ to \eqref{SNLW7a}
with the zero initial data,  
where $\{\,\wick{\Phi^\ell}\,\}_{\l = 1}^k$
denotes the stochastic convolution $\Phi$ defined in~\eqref{phi2}
and its Wick powers 
$\wick{\Phi^\ell}$, $\l = 2, \dots, k$,  
defined in Lemma \ref{LEM:Psi} below.

\end{theorem}

The proof of Theorem \ref{THM:GWP2} is based on 
Bourgain's invariant measure argument \cite{BO94, BO96}.
By viewing the  SvNLW dynamics \eqref{SNLW9} as the ``superposition'' 
of the (renormalized) NLW dynamics \eqref{NLW5}
and the Ornstein-Uhlenbeck dynamics (for $\dt u$):
\begin{align*}
\dt (\dt  u) = -  D \dt u 
+ \sqrt{2} D^\frac{1}{2}  \xi, 
\end{align*}

\noi
we expect the Gibbs measure
(for the standard  NLW \eqref{NLW5}), formally given by
\begin{align}
``d\rhoo(u,\dt u ) = Z^{-1}e^{-E( u, \dt u  )}dud(\dt u)", 
\label{Gibbs1}
\end{align}

\noi
to be invariant 
under the SvNLW dynamics \eqref{SNLW9}.

Let $\muu_1$ be  the induced probability measure under the map:
$\o \in \O \longmapsto (u_0^\o, u_1^\o)$,
where 
 $(u_0^\o, u_1^\o)$ is as  in \eqref{IV1}.
Then, we can write $\muu_1$ as $\muu_1 = \mu_1\otimes \mu_0$, 
where  $\mu_s$ denote
a Gaussian measure on periodic distributions,   formally defined by
\begin{align}
 d \mu_s 
   = Z_s^{-1} e^{-\frac 12 \| u\|_{{H}^{s}}^2} du
& =  Z_s^{-1} \prod_{n \in \Z^2} 
 e^{-\frac 12 \jb{n}^{2s} |\ft u(n)|^2}   
 d\ft u(n) .
\label{gauss0}
\end{align}

\noi
Note that $\mu_1$ corresponds to the massive Gaussian free field, 
while $\mu_0$ corresponds to the white noise.
Then, 
by renormalizing the potential part of the energy $ E(\vec u)$ in \eqref{E1}, 
we can indeed construct  the  Gibbs measure $\rhoo$ 
as a probability measure
such that $\rhoo$ and $\muu_1$
are mutually absolutely continuous.
By exploiting the formal invariance of
the Gibbs measure $\rhoo$
under \eqref{SNLW9}, 
Bourgain's invariant measure argument yields
almost sure global well-posedness of \eqref{SNLW9}
with respect to the Gibbs measure $\rhoo$; 
see Theorem \ref{THM:GWP3} below.
By invoking the mutual absolute continuity 
of $\rhoo$ and $\muu_1$, 
we then conclude Theorem \ref{THM:GWP2}.
  See
Subsection \ref{SUBSEC:GWP2a} for details.

As a corollary to Theorem \ref{THM:GWP2}
and the Cameron-Martin theorem \cite{CM}, we obtain the following almost sure global
well-posedness of \eqref{SNLW9}
with 
deterministic $\H^1$-initial data $(v_0, v_1)$ perturbed by 
the random functions $(u_0^\o, u_1^\o)$ in \eqref{IV1}.
Let us introduce some notations.
{\it Fix} $\vec v_0  = (v_0, v_1) \in \H^1(\T^2)$.
With 
 $\vec u_0^\o = (u_0^\o,u_1^\o)$, 
define the stochastic convolution 
$\Phi[\vec v_0 + \vec u_0^\o]$
with the shifted initial data
$\vec v_0 + \vec u_0^\o$:
\begin{align} 
\begin{split}
\Phi[\vec v_0 + \vec u_0^\o](t) 
&  = V(t)(\vec v_0 + \vec u_0^\o)
+   \sqrt{2} \int_0^t S(t-t') D^\frac{1}{2} dW(t')\\
&   
 = V(t)\vec v_0 + \Phi (t), 
\end{split}
\label{phi3}
\end{align}

\noi
where $\Phi$ is as in \eqref{phi2}.
Given $N \in \N$,  set $\Phi_N[\vec v_0 + \vec u_0^\o]
= \P_N \Phi[\vec v_0 + \vec u_0^\o]$.
In view of 
\begin{align*}
H_\l(x+y; \s )
&  = 
\sum_{j = 0}^\l
\begin{pmatrix}
\l \\ j
\end{pmatrix}
 x^{\l - j} H_j(y; \s)
\end{align*}

\noi
and \eqref{Wick3}, 
we define the Wick power $:\! (\Phi_N[\vec v_0 + \vec u_0^\o])^\l  \!:$ by 
\begin{align}
\begin{split}
:\! (\Phi_N[\vec v_0 + \vec u_0^\o])^\l ( t, x) \!:
 \, &  \deff  H_{\l} (\Phi_N[\vec v_0 + \vec u_0^\o]( t, x); \al_N)\\
& \, =  \sum_{j = 0}^\l
\begin{pmatrix}
\l \\ j
\end{pmatrix}
(V(t)\vec v_0)^j :\!\Phi_N^{\l-j} ( t, x) \!:.
\end{split}
\label{Wick4}
\end{align}

\noi
where $\al_N$ is as in \eqref{sig2}
and $:\!\Phi_N^{\l-j}  \!: \ = \ :\!(\Phi_N[\vec u_0^\o])^{\l-j}  \!:$
is as in \eqref{Wick3}.
Thanks to the $H^1$-regularity of $V(t)\vec v_0$
and the almost sure convergence of $:\!\Phi_N^{\l-j}  \!:$,  
 we see that  the Wick power $\wick{(\Phi_N[\vec v_0 + \vec u_0^\o])^\ell}$ converges to a limit $\wick{(\Phi[\vec v_0 + \vec u_0^\o])^\ell}$ in $C([0,T];W^{-\eps,\infty} (\T^2))$ for any $\eps > 0$ and $T > 0$, almost surely;
see Subsection~\ref{SUBSEC:GWP2b}.

Proceeding as in Subsection 
\ref{SUBSEC:1.2}
with the first order expansion 
$u_N = \Phi_N[\vec v_0 + \vec u_0^\o] + v_N$
and taking $N \to \infty$, 
we can reformulate  the  renormalized SvNLW \eqref{SNLW9}
with the shifted initial data
$(u, \dt u)|_{t = 0} = \vec v_0 + \vec u_0^\o$
as
\begin{equation}
\begin{cases}
\dt^2 v + (1-\Dl) v + D \dt v + \sum_{\ell = 0}^k \binom{k}{\ell} \wick{(\Phi
[\vec v_0 + \vec u_0^\o])^\ell} v^{k-\ell} = 0\\
(v, \dt v)|_{t = 0} = (0, 0).
\end{cases}
\label{SNLW7b}
\end{equation}

\begin{corollary}\label{COR:GWP3}

Let $k \in 2 \N + 1$ and fix $(v_0, v_1) \in \H^1(\T^2)$.
Then, the renormalized SvNLW~\eqref{SNLW9}
is almost surely globally well-posed 
with respect to the shifted initial data
 $(v_0, v_1) + (u_0^\o, u_1^\o)$,
 where
 $(u_0^\o, u_1^\o)$ is as in  \eqref{IV1}, 
in the sense that 
there exists almost surely a unique global-in-time solution $v$ to \eqref{SNLW7b}
with the zero initial data.

\end{corollary}

In 
Subsection \ref{SUBSEC:GWP2b}, 
we sketch the argument.
See
 \cite{OQ} for a further discussion
 on 
 probabilistic well-posedness and other aspects
 (such as  the large deviation principle)
 for 
 random initial data 
 of the form ``a smooth deterministic function + a rough random perturbation''.

 \medskip
 
 We conclude this introduction by stating several remarks.

\begin{remark}\rm 

(i) 
The pathwise global well-posedness result (Theorem \ref{THM:GWP1})
relies on a dispersive PDE argument.
As such, 
the coefficient $\sqrt{2}$ on  the noise $D^\frac 12 \xi$ in the equation~\eqref{SNLW9}
plays no role
and thus 
Theorem 
\ref{THM:GWP1}
also applies to the cubic SvNLW \eqref{SNLW9} (with $k = 3$)
with a general coefficient on the noise.
On the other hand, 
Theorem \ref{THM:GWP2} relies
on the invariant measure argument
and thus the coefficient on the  noise in  \eqref{SNLW9}
must be $\sqrt{2}$.\footnote{For general $0 < \mu < 1$, 
Theorem \ref{THM:GWP2} also holds for the following equation:
\begin{align*}
\dt^2 u + (1-\Dl) u + 2\mu D \dt u \, \, +  \wick{u^k}\, 
= 2\sqrt{\mu} D^\frac{1}{2}  \xi, 
\end{align*}

where the   coefficient on the  noise is $2\sqrt\mu$.\label{FT:mu}}


At this point, we do not know how to extend the pathwise global well-posedness
(Theorem~\ref{THM:GWP1}) to the (super-)quintic case.
Even with a smoother noise, 
one would need to use a trick 
introduced in 
\cite{OP1, Lat} to handle the higher homogeneity.
See  \cite{Liu}.
%
%
It may also be  of interest to investigate if
a parabolic PDE approach such as those in \cite{MW2, Tren}
can be applied to handle the (super-)quintic case.

\smallskip

\noi
(ii) As mentioned in Remark \ref{REM:Gibbs}, 
Theorems \ref{THM:LWP} and \ref{THM:GWP1}
apply to~\eqref{SNLW1} with $\al = \frac 12$
with essentially identical proofs.
When $\al < \frac 12$, the equation \eqref{SNLW1} is no longer singular
(namely, a solution is a function and there is no need 
for a renormalization).
See~\cite{Liu} for pathwise global well-posedness results 
for higher values of $k \in 2\N + 1$, 
when $ \al < \frac 12$.

\smallskip

\noi
(iii) 
When $k = 2$, the equation \eqref{SNLW9} is no longer defocusing.
Even in this case, however, the Gibbs measure can be constructed;
see \cite{BOpark, OST}
and thus an analogue of almost sure global well-posedness
(Theorem \ref{THM:GWP2})
holds  when $k = 2$.
In the non-defocusing case with $k \geq 3$, 
namely, either (i) even $k \geq 4$ or (ii) with the nonlinearity $-\wick{u^k}\,$, $k \in 2\N+1$, 
(i.e.~with the negative sign)
in \eqref{SNLW9}, 
it is known that the Gibbs measure is not constructible
\cite{BS, OST}
and hence Bourgain's invariant measure argument is not applicable in this case.


\smallskip

\noi
(iv)  For the physical reason, 
it is of interest to investigate the well-posedness issue of \eqref{SNLW9} on 
$\R^2$.  In this case, due to the unboundedness of the domain, 
the integrability becomes an issue.
In view of the pathwise global well-posedness
(Theorem \ref{THM:GWP1}), we expect that the dispersive techniques as in \cite{Tol}
may be applied to treat the cubic case.
As for the higher order nonlinearity, it may be possible to 
adapt the parabolic approach as in \cite{MW2}.
We plan to address this issue in a forthcoming work.

\smallskip

\noi
(v) As for the model \eqref{SNLW3} without renormalization, 
we expect 
a {\it triviality} result to hold.
Roughly speaking, 
  extreme oscillations 
  make 
solutions $u_N$ to~\eqref{SNLW5} with regularized noises
tend to a solution to the  linear stochastic viscous  wave equation~\eqref{SNLW4}
(or the trivial solution) as the regularization is removed.
Such a triviality result (in the absence of renormalization) is known 
for  stochastic NLW and stochastic nonlinear heat equations; see \cite{russo4, HRW, OOR, ORSW}.

\end{remark}

\begin{remark} \rm
From its derivation, 
the viscous wave equation is most relevant physically 
in two spatial dimensions;
see \cite{KC1, CKO}.
At the same time, it is of interest to study the equation in other spatial dimensions.

\smallskip

\noi
(i) 
Let us first consider the following equation on $\T^d$:
\begin{align}
\dt^2 u +(1- \Dl)  u  + D \dt u  + u^k  = D^\al \xi.
\label{Y1}
\end{align}

\noi
When $\al < \frac{3-d}{2}$, 
it follows from
the $\frac 32$-smoothing of the viscous wave operator
 that the stochastic convolution $\Psi$ defined by  
\begin{equation}
\Psi (t) =   \int_0^t S(t-t')  D^\al dW(t')
\label{psi5}
\end{equation}

\noi
is a function and hence there is no need for renormalization to study \eqref{Y1}.
See a recent preprint \cite{Liu}.

When $\al = \frac {3-d}{2}$, the stochastic convolution $\Psi$ in \eqref{psi5}
has regularity slightly below 0
and we need to apply the Wick renormalization to study the equation.
When $d = 1$, the local well-posedness for general $k \ge2$
and the pathwise global well-posedness in the cubic case ($k = 3$)
as in Theorems \ref{THM:LWP} and \ref{THM:GWP1}, respectively, 
hold true with the same proofs.
When $d = 3$ (corresponding to the 
space-time white noise forcing), 
%
%
%
%
%
%
in view of  the embedding $H^1(\T^3) \subset L^6(\T^3)$, 
a slight modification of the proofs
of 
Theorems~\ref{THM:LWP} and \ref{THM:GWP1}
shows that 
(a) for  $k = 2, 3$,  (the renormalized version of) SvNLW \eqref{Y1} (with $\al  = 0$)
on $\T^3$ is locally well-posed in $\H^s(\T^3)$, $ s\ge 1$, 
and (b) 
(the renormalized version of) the cubic  SvNLW \eqref{Y1} (with $k = 3$ and $\al = 0$)
on $\T^3$ is globally well-posed in $\H^s(\T^3)$, $ s\ge 1$.
Due to the more restrictive range of Sobolev's inequality on $\T^3$, 
however, 
the proof of Theorem \ref{THM:LWP} does not apply 
to SvNLW \eqref{Y1} (with $\al = 0$) on $\T^3$ 
for  $k \geq 4$.
In this case, one needs to make use of the (wave) Strichartz estimates
to prove local well-posedness.
In  higher dimensions, 
one also needs to use the Strichartz estimates (except for $d = 4$ and $k = 2$,
which can be handled by Sobolev's embedding).
We, however, do not pursue this issue in this paper.

When $\al >  \frac {3-d}{2}$, the stochastic convolution $\Psi$ in \eqref{psi5}
has even lower regularity, possibly requiring a further renormalization; see Part (ii) below.
For values of $\al$ close to $\frac {3-d}{2}$, the proof of local well-posedness
(Theorem \ref{THM:LWP}) is applicable but the value of $\al$ depends on the degree $k$
of the nonlinearity.
For higher values of $\al$, one needs to use a more sophisticated approach
such as the paracontrolled approach  \cite{GKO2, OOTol, OOTol2, Bring2} together with the Strichartz estimates.
As for the global well-posedness, the proof of Theorem \ref{THM:GWP1}
crucially exploits the logarithmic divergence
of the stochastic convolution and thus it is not applicable to the case
$\al >  \frac {3-d}{2}$.

\smallskip

\noi
(ii) Next, we consider the following equation on $\T^d$:
\begin{align}
\dt^2 u +(1- \Dl)  u  + D \dt u  + u^k  = \sqrt 2 D^\frac{1}{2} \xi
\label{Y2}
\end{align}

\noi
with $k \in 2 \N + 1$, where one may prove almost sure global well-posedness
via Bourgain's invariant measure argument.
When $ d= 1$, the stochastic convolution $\Phi$ defined in \eqref{phi2}
has spatial regularity $\frac 12 - \eps$
and thus there is no need to introduce renormalization.
In this case, local well-posedness easily follows from Sobolev's inequality (without
the first order expansion),  and Bourgain's invariant measure argument
\cite{BO94} yields the one-dimensional analogue of Theorem \ref{THM:GWP2}.

When $d = 3$, we first recall that the Gibbs measure $\rhoo$ in \eqref{Gibbs1}
and the Gaussian measure $\muu_1 = \mu_1 \otimes \mu_0$ (= the distribution
of the random initial data $(u_0^\o, u_1^\o)$ in \eqref{IV1})
are mutually singular; see \cite{BG1}.
Hence, we need to study the Gibbsian initial data in this case.
Noting that the Gibbs measure $\rhoo$ corresponds to 
the $\Phi^4_3$-measure (on $u$) $\otimes$ the spatial white noise (on $\dt u$), 
we see that the Wick renormalization
is not sufficient and that we need to introduce another renormalization to remove
the logarithmic divergence; see \cite{Hairer, MWX} 
in the case of the parabolic $\Phi^4_3$-model.
The well-posedness theory in this case
certainly requires 
 a more sophisticated approach
such as the paracontrolled approach  \cite{GKO2},
which is beyond the scope of this paper.

\end{remark}

\section{Preliminary lemmas}

\subsection{Tools from stochastic analysis}
\label{SUBSEC:sto}

For readers' convenience, we first recall the Hermite polynomials $H_k(x; \s)$, 
 defined through the following generating function:
\begin{equation*}
F(t, x; \s) = e^{tx - \frac{1}{2}\s t^2} = \sum_{k = 0}^\infty \frac{t^k}{k!} H_k(x;\s), 
 \end{equation*}
	
\noi
which is used in constructing the renormalized powers
\eqref{Wick1}, \eqref{Wick3}, and \eqref{Wick4}.

The following lemma
 establishes the regularity 
of the stochastic convolution $Z$
and its Wick powers $\wick{Z^k}\, $, 
where $Z= \Psi$ in \eqref{psi1} or $Z = \Phi$ in \eqref{phi2}.

\begin{lemma}\label{LEM:Psi}
Let $Z = \Psi$ or  $\Phi$.
Given  $k \in \N$
and $N \in \N$, let $:\! Z_N^k \!:
\, = \, :\! (\P_N Z)^k \!:$
denote the truncated Wick power defined
in \eqref{Wick1} or \eqref{Wick3}, respectively.
Then, 
given any  $T,\eps>0$ and finite $p \geq 1$, 
 $\{ \, :\! Z_N^k \!: \, \}_{N\in \N}$ is a Cauchy sequence in $L^p(\O;C([0,T];W^{-\eps,\infty}(\T^2)))$,
 converging to some limit $:\!Z^k\!:$ in $L^p(\O;C([0,T];W^{-\eps,\infty}(\T^2)))$.
Moreover,  $:\! Z_N^k \!:$  converges almost surely to the same  limit in $C([0,T];W^{-\eps,\infty}(\T^2))$.
Furthermore, we have the following tails estimates.

\smallskip

\noi
\textup{(i)}
Given any finite $q\geq 1$, we   have 
\begin{align}
P\Big( \|:\! Z^k \!:\|_{L^q_T W^{-\eps, \infty}_x} > \ld\Big) 
\leq C\exp\bigg(-c \frac{\ld^{\frac{2}{k}}}{T^{ \frac{2}{qk}}}\bigg)
\label{P0}
\end{align}

\noi
for any $T \geq 1$ and $\ld > 0$.

\smallskip

\noi
\textup{(ii)}
When $q = \infty$, we  have 
\begin{align}
P\Big( \|:\! Z^k \!:\|_{L^\infty ([j, j+1]; W^{-\eps, \infty}_x)}> \ld\Big) 
\leq C\exp\bigg(-c \frac{\ld^{\frac{2}{k}}}{j+1}\bigg)
\label{P0z}
\end{align}

\noi
for any $j \in \Z_{\ge 0}$ and $\ld > 0$.

\smallskip

\noi
\textup{(iii)}
When $q = \infty$ and $k= 1$, we  have 
\begin{align}
P\Big( \| Z \|_{L^\infty ([0, T]; W^{-\eps, \infty}_x)}> \ld\Big) 
\leq CT \exp\bigg(-c \frac{\ld^2\eps}{T}\bigg)
\label{P0z2}
\end{align}

\noi
for any $T\geq 1$ and $\ld > 0$.

\end{lemma}

Part (iii) of this lemma in particular plays an important role
in the proof of Theorem~\ref{THM:GWP1}.
See Section \ref{SEC:GWP1}.

\begin{proof}
In view of \eqref{sig1} and \eqref{sig2}, 
the proof of this lemma is essentially identical 
to 
that of Lemma 2.3 in \cite{GKOT} for the stochastic (damped) wave equation.
Hence, we will be very brief here.
As for the convergence part of the statement, 
see \cite[Proposition~2.1]{GKO} 
and  \cite[Lemma~3.1]{GKO2}.

\smallskip

\noi
(i) 
As for the exponential tail estimate \eqref{P0}, 
by repeating the argument
in the proof of \cite[Proposition~2.1]{GKO}, 
we have
\begin{align}
\E\big[| \jb{\nb}^{-\eps} :\! Z^{k} (t, x) \!:  |^2 \big] 
& \les \sum_{n_1, \dots, n_{k} \in \mathbb{Z}^2} 
\frac{1}{\langle n_1
\rangle^2 \cdots \langle n_{k} \rangle^2 \langle n_1 + \cdots + n_{k}
\rangle^{2\eps}} \le C_\eps 
\label{P0z1}
\end{align}

\noi
for any $\eps > 0$, uniformly in 
$x \in \T^2$ and $t \geq 0$.
Then, given finite $q \geq 1$, Sobolev's inequality (with some $r > 4\eps^{-1}$),  Minkowski's integral inequality,  and the Wiener chaos estimate 
(\cite[Lemmas 2.3 and 2.4]{GKO}) yield
\begin{align}
 \Big\| \| :\! Z^{k}  \!:\|_{L^q_T W^{-\eps, \infty}_x}\Big\|_{L^p(\O)}
\les 
  \Big\| \| :\! Z^{k}  \!:\|_{L^q_T W^{-\frac \eps 2, r}_x}\Big\|_{L^p(\O)}
\les p ^\frac{k}{2} T^{ \frac{1}{q}}
\label{P0a}
\end{align}

\noi
for any  $p \geq \max (q, r) $.
Then, the bound  \eqref{P0}
follows from \eqref{P0a} and Chebyshev's inequality
(as in \cite[Lemma 3]{BOP1}).\footnote{Lemma 2.2 in the arXiv version.}

\smallskip

\noi
(ii) 
As for the second bound \eqref{P0z}, we first write 
\begin{align}
\begin{split}
P\Big( \|:\! Z^k \!:\|_{L^\infty ([j, j+1]; W^{-\eps, \infty}_x)}
& > \ld\Big) 
\leq
 P\Big( \|:\! Z^k(j) \!:\|_{ W^{-\eps, \infty}_x}> \tfrac{\ld}{2}\Big) \\
+ 
& P\Big( \sup_{t \in [j, j+1]}\|:\! Z^k(t) \!: - :\! Z^k(j) \!:\|_{ W^{-\eps, \infty}_x}> \tfrac \ld2\Big)
\end{split}
\label{P0b}
\end{align}

\noi
for given $j \in \Z_{\ge 0}$ and $\ld > 0$. 
Using \eqref{P0z1}, we can repeat the argument in Part (i)
to bound 
the first term on the right-hand side of \eqref{P0b}.
As for the second term on the right-hand side
of~\eqref{P0b}, 
we first recall from  the proof of \cite[Proposition~2.1]{GKO}
that 
\begin{align*}
\Big\| |h|^{-\rho} \|\dl_h (:\! Z^{k} (t) \!:) \|_{W^{-\eps, \infty}_x} \Big\|_{L^p(\O)}
\les p^\frac{k}{2} (j+1)^\frac{k}{2}
\end{align*}

\noi
for any sufficiently large $p \gg1 $, $t \in [j, j+1]$, 
and $|h| \leq 1$, where $\dl_h f(t) = f(t+h) - f(t)$ and $0 < \rho < \eps$.
Then, the desired bound \eqref{P0z} follows
from Chebyshev's inequality and 
the Garsia-Rodemich-Rumsey inequality
(\cite[Theorem A.1]{FV} and \cite[Lemma 2.2]{GKOT}), 
which provides an exponential tail bound
for a H\"older norm (in time).
See the proof of Lemma~2.3 in \cite{GKOT} for further details.

\smallskip

\noi
(iii) 
The third bound \eqref{P0z2} follows in a similar manner once 
we make $\eps$-dependence more explicit.
By Sobolev's inequality,   Minkowski's integral inequality,  and the Wiener chaos estimate (\cite[Lemmas 2.3 and 2.4]{GKO}), we have 
\begin{align*}
\begin{split}
\big\| \| Z (j) \|_{W_x^{-\eps,\infty}} \big\|_{L^p (\O)} 
& \les \big\| \| Z (j) \|_{W_x^{-\frac \eps2,r}} \big\|_{L^p (\O)} 
 \les p^{\frac 12} \Big\| \big\| \jb{\nb}^{-\frac{\eps}{2}} Z (j) \big\|_{L^2 (\O)} \Big\|_{L_x^r} \\
& \les p^{\frac 12} \bigg\| \Big( \sum_{n \in \Z^2} \frac{1}{\jb{n}^{2+\eps}} \Big)^{\frac 12} \bigg\|_{L_x^r} \\
& \sim p^{\frac 12} \eps^{-\frac 12}
\end{split}
\end{align*}

\noi
for any $p \geq r > 4\eps^{-1}$.
Similarly, by the mean value theorem (with \eqref{psi1} or \eqref{phi2}), we have 
\begin{align*}
\Big\| |h|^{-\rho} \|\dl_h Z (t) \|_{W^{-\eps, \infty}_x} \Big\|_{L^p(\O)}
 \les p^{\frac 12} |h|^{\frac{\eps}{4}-\rho}\bigg( \sum_{n \in \Z^2} \frac{1}{\jb{n}^{2+\eps}} 
 \bigg)^{\frac 12} 
\les p^\frac{1}{2}\eps^{-\frac{1}{2}}
\end{align*}

\noi
for any sufficiently large $p \gg 1$, $t,t+h \in [0,T]$, 
and $|h| \leq 1$, provided that  $0 < \rho < \frac \eps4$.
Then, the rest follows from proceeding as in Part (ii)
and summing over the interval $[j, j+1]$.
\end{proof}

\subsection{Tools from deterministic analysis}

We first recall the product estimates.  

\begin{lemma}\label{LEM:bilin}
 Let $0\le s \le 1$.

\smallskip

\noi
\textup{(i)} Suppose that 
 $1<p_j,q_j,r < \infty$, $\frac1{p_j} + \frac1{q_j}= \frac1r$, $j = 1, 2$. 
 Then,  we have 
\begin{equation*}  
\| \jb{\nb}^s (fg) \|_{L^r(\T^d)} 
\les \Big( \| f \|_{L^{p_1}(\T^d)} 
\| \jb{\nb}^s g \|_{L^{q_1}(\T^d)} + \| \jb{\nb}^s f \|_{L^{p_2}(\T^d)} 
\|  g \|_{L^{q_2}(\T^d)}\Big).
\end{equation*}

\smallskip

\noi
\textup{(ii)} 
Suppose that  
 $1<p,q,r < \infty$ satisfy 
$\frac1p+\frac1q\leq \frac1r + \frac{s}d $.
Then, we have 
\begin{align*}
\big\| \jb{\nb}^{-s} (fg) \big\|_{L^r(\T^d)} \les \big\| \jb{\nb}^{-s} f \big\|_{L^p(\T^d) } 
\big\| \jb{\nb}^s g \big\|_{L^q(\T^d)}.  
\end{align*}

\end{lemma} 

See \cite{GKO} for  the proof.
Note that
while  Lemma \ref{LEM:bilin} (ii) 
was shown only for 
$\frac1p+\frac1q= \frac1r + \frac{s}d $
in \cite{GKO}, 
the general case
$\frac1p+\frac1q\leq \frac1r + \frac{s}d $
follows from the inclusion $L^{r_1}(\T^d)\subset L^{r_2}(\T^d)$
for $r_1 \geq r_2$.

\smallskip

Next, we state  a Schauder-type estimate for the Poisson kernel 
\begin{align*}
P(t) = e^{-\frac{D}{2}t}.
\end{align*}

\begin{lemma}\label{LEM:Sch}
Let $1 \leq p \leq q \leq \infty$ and $\al \geq 0$. Then, we have
\begin{equation}
\| D^\al P(t) f \|_{L^q(\T^d)} \les t^{-\al-d(\frac{1}{p}-\frac{1}{q})} \| f \|_{L^p(\T^d)}
\label{Sch1}
\end{equation}

\noi
for any $0 < t \leq 1$. 

\end{lemma}
\begin{proof}
We first prove \eqref{Sch1} on $\R^d$ for any $t > 0$. Let $K_t(x)$ denote the kernel for $P(t)$ given by 
$\F_{\R^d} (K_t)(\xi) = e^{-\frac{|\xi|}{2}t}$, 
where $\F_{\R^d}$ denotes the Fourier transform on $\R^d$.
Recall that 
\begin{align}
 K_t(x) =  c_d \frac{  t}{(t^2+|x|^2)^{\frac{d+1}2}}. 
 \label{K0a}
\end{align}

\noi
Noting
$D^\al P(t) f = (D^\al K_t) *f$, we need to study the scaling property of $D^\al K_t$.
On the Fourier side, we have
\[\F_{\R^d} (D^\al K_t)(\xi) = |\xi|^{\al} e^{-\frac{|\xi|}{2}t}
= t^{-\al}  (t |\xi|)^{\al} e^{-\frac{|\xi|t}{2}}
= t^{-\al} \F_{\R^d}(D^\al K_1)(t\xi).
\]

\noi
Namely, we have 
\begin{align}
D^\al K_t(x)  = t^{-d-\al} (D^\al K_1)(t^{-1}x).
\label{K0}
\end{align}

\noi
For $1 \leq r \leq \infty$ with $\frac{1}{r} = \frac{1}{q} - \frac{1}{p} + 1$, 
from \eqref{K0} and \eqref{K0a}, we have\footnote{One may first
use \eqref{K0a} to  show \eqref{K1} for even $\al \geq 0$
and then interpolate the result to  deduce \eqref{K1} for general $\al \geq 0$.}
\begin{align}
\| D^\al K_t\|_{L^r (\R^d)} =t^{-\al -d (1 - \frac{1}{r})} 
\|D^\al  K_1\|_{L^r (\R^d)}
= C_{r, \al} t^{-\al -d (\frac{1}{p} - \frac{1}{q})} .
\label{K1}
\end{align}

\noi
Then, \eqref{Sch1} follows from Young's inequality and \eqref{K1}.

Next, we prove \eqref{Sch1} on $\T^d$.
 Let $R_t(x)$ denote the kernel for $P(t)$ on $\T^d$ given by $\ft R_t (n) = e^{-\frac{|n|}{2}t}
=  \F_{\R^d}(K_t)(n) $. 
Then, given any
 $1 \leq r \leq \infty$, from 
the 
Poisson summation formula (with \eqref{K0a}) and H\"older's inequality,
we have 
\begin{align}
\begin{split}
\| D^\al R_t \|_{L^r(\T^d)} 
&= \bigg\| \sum_{n \in \Z^d}  \F_{\R^d} (D^\al K_t) (n) e^{in \cdot x} \bigg\|_{L^r(\T^d)} 
= \bigg\| \sum_{n \in \Z^d} D^\al  K_t (x+n) \bigg\|_{L^r(\T^d)} \\
&\leq \bigg\| \bigg( \sum_{n \in \Z^d} \jb{n}^{-\beta r'} \bigg)^{\frac1{r'}} 
\big\| \jb{n}^\beta D^\al K_t(x+n) \big\|_{\ell_n^r} \bigg\|_{L_x^r(\T^d)} \\
&\les \big\| \jb{x}^\beta D^\al K_t(x) \big\|_{L^r(\R^d)} \\
&\les \big\| D^\al K_t(x) \big\|_{L^r(\R^d)} + \big\| |\tfrac{x}{t}|^\beta D^\al K_t(x) \big\|_{L^r(\R^d)},
\end{split}
\label{K2}
\end{align}

\noi
provided that $\beta > d(1-\tfrac{1}{r})$ and $0 < t \leq 1$. 
From \eqref{K0}, we have 
\begin{align}
 \big\| |\tfrac{x}{t}|^\beta D^\al K_t(x) \big\|_{L^r(\R^d)} 
 = t^{-\al -d (1 - \frac{1}{r})} \big\| |x|^\beta D^\al  K_1(x) \big\|_{L^r(\R^d)} = C_{r, \al}' t^{-\al -d (1 - \frac{1}{r})}
\label{K3}
 \end{align}
 
\noi
for some finite $C_{r, \al}'  > 0$,  
provided that $\beta < d(1-\tfrac{1}{r}) + 1$. 
Hence, the desired bound \eqref{Sch1} follows
from Young's inequality and \eqref{K2} with \eqref{K1} and \eqref{K3}.
\end{proof}

\section{Local well-posedness of the stochastic viscous NLW}

In this section, we present the proof of Theorem \ref{THM:LWP}.
For this purpose, we consider
the following deterministic vNLW:
\begin{align}
\begin{cases}
\dt^2 v + (1-\Dl)v + D \dt v + \sum_{\ell = 0}^k \binom{k}{\ell} \, \Xi_\ell\,  v^{k-\ell}
= 0\\
(v,\dt v)|_{t = 0}=(v_0,v_1)
\end{cases}
\label{rNLW2}
\end{align}

\noi
for given initial data $(v_0,v_1) \in \H^s(\T^2)$, $s \geq 1$,   and 
deterministic source terms  $(\Xi_0,\dots,\Xi_k)$ with the understanding that $\Xi_0 \equiv 1$.
Define $\mathcal{X}^s (\T^2)$ by
\begin{align}
 \mathcal{X}^s (\T^2) = \H^s (\T^2) \times
\prod_{\l = 0}^{k-1}  C([0,1]; W^{-\frac{\l}{k},\infty}(\T^2)) 
\times 
 C([0,1]; W^{-1+\eps,\infty}(\T^2)) 
\label{enh1}
\end{align}

\noi
for some small $\eps > 0$
and set
\[ \|\pmb{\Xi}\|_{\mathcal{X}^s}
= \| (v_0, v_1) \|_{\H^s} + \sum_{\l = 1}^{k-1} \| \Xi_\l\|_{C([0, 1]; W^{-\frac{\l}{k}, \infty})} 
+ \| \Xi_k\|_{C([0, 1]; W^{-1+\eps, \infty})} \]

\noi
for $\pmb{\Xi} = (v_0, v_1, \Xi_1, \Xi_2, \dots, \Xi_k) \in \mathcal{X}^s (\T^2)$. Then, we have the following local well-posedness result for \eqref{rNLW2}.

\begin{proposition}
\label{PROP:LWP2}
Let $k\geq 2$ be  an integer and $s \geq 1$.
Then,  \eqref{rNLW2} is locally well-posed in $\mathcal{X}^s (\T^2)$. More precisely, given an enhanced data set:
\begin{equation*}
\pmb{\Xi} = (v_0, v_1, \Xi_1, \Xi_2, \dots, \Xi_k) \in \mathcal{X}^s (\T^2),
\end{equation*}

\noi
there exist $0 < T = T(\| \pmb{\Xi} \|_{\mathcal{X}^s}) \leq 1$ and a unique solution 
$\vec v = (v, \dt v)  \in C([0,T]; \H^1 (\T^2))$
 to~\eqref{rNLW2}, 
 depending  continuously on the enhanced data set $\pmb{\Xi}$.
\end{proposition}

Note that 
Proposition~\ref{PROP:LWP2}
is completely deterministic.
Theorem \ref{THM:LWP}
immediately follows from 
 Proposition~\ref{PROP:LWP2} 
and 
Lemma \ref{LEM:Psi}
which states
 that 
the (random) enhanced data set
$\pmb{\Xi} = (v_0, v_1,  \Psi, :\!\Psi^2\!:, \dots, :\!\Psi^k\!:\, )$
almost surely belongs to 
 $\mathcal{X}^{s}(\T^2)$
 and that the 
the truncated (random) enhanced data set
$\pmb{\Xi}_N = (v_0, v_1,  \Psi_N, :\!\Psi_N^2\!:, \dots, :\!\Psi_N^k\!:\, )$
converges 
almost surely  to $\pmb{\Xi}$ in 
 $\mathcal{X}^{s}(\T^2)$.

\begin{proof}
By writing \eqref{rNLW2} in the Duhamel formulation, we have
\begin{equation}
v(t) = \G (v)\deff V(t) (v_0,v_1) - \sum_{\ell = 0}^k \binom{k}{\ell} \int_0^t S(t-t') \big(\Xi_\ell \, v^{k-\ell}\big)(t') dt',
\label{L1}
\end{equation}

\noi
where $V(t)$ and $S(t)$ are as in \eqref{S2} and \eqref{S1} respectively.
Let $\vec \G(v) = (\G(v), \dt \G(v))$.

 Fix $0 < T \leq 1$.
We first consider  the case $\ell = 0$. From \eqref{S1}  and Sobolev's inequality, we have
\begin{align}
\begin{split}
\bigg\| \int_0^t &  S(t-t') v^k (t')  dt' \bigg\|_{C_T H_x^1} 
 + \bigg\|\dt  \int_0^t S(t-t') v^k (t') dt' \bigg\|_{C_T L^2_x} \\
&\les \bigg\| \int_0^t \big\| \sin( (t-t')\jbb{D}) v^k (t') \big\|_{L_x^2} dt' \bigg\|_{C_T} \\
&\les T \big\| v^k \big\|_{C_T L_x^2}  \les T \| v \|_{C_T L_x^{2k}}^k \\
&\les T \| v \|_{C_T H_x^1}^k.
\end{split}
\label{L2}
\end{align}

\noi
Next, let  $1 \leq \ell \leq k-1$. 
From Lemma \ref{LEM:Sch}, Lemma \ref{LEM:bilin}\,(ii) and then (i) followed by Sobolev's inequality, 
we have 
\begin{align}
\begin{split}
\bigg\| & \int_0^t S(t-t')  \big( \Xi_\ell \, v^{k-\ell} \big) (t') dt' \bigg\|_{C_T H_x^1} 
+ \bigg\| \dt \int_0^t S(t-t') \big( \Xi_\ell \, v^{k-\ell} \big) (t') dt' \bigg\|_{C_T L^2_x}\\ 
&\les T^{\frac{k-\l}{k}} \big\| \Xi_\ell \, v^{k-\ell} \big\|_{C_T H_x^{-\frac{\l}{k}}} 
\les T^{\frac{k-\l}{k}} \big\| \jb{\nb}^{-\frac{\l}{k}} \Xi_\ell \big\|_{C_T L_x^{2k}}
 \big\| \jb{\nb}^{\frac{\l}{k}} v^{k-\ell} \big\|_{C_T L_x^\frac{2k}{k+\l-1}} \\
&\les T^{\frac{k-\l}{k}} \| \pmb{\Xi} \|_{\mathcal{X}^s}
 \big\| \jb{\nb}^{\frac{\l}{k}} v \big\|_{C_T L_x^{\frac{2k}{\l}}}
\| v \|_{C_T L_x^{\frac{2k(k-\l-1)}{k-1}} }^{k-\l-1}\\
&\les T^{\frac{k-\l}{k}} \| \pmb{\Xi} \|_{\mathcal{X}^s} \| v \|_{C_T H_x^1}^{k-\ell}.
\end{split}
\label{L3}
\end{align}

\noi
Lastly, when $\l  = k$, it follows from \eqref{S1} and Lemma \ref{LEM:Sch} that 
\begin{equation}
\begin{split}
\bigg\| &  \int_0^t S(t-t') \Xi_k (t') dt' \bigg\|_{C_T H_x^1} 
+ \bigg\| \dt \int_0^t S(t-t') \Xi_k (t') dt' \bigg\|_{C_T L^2_x} \\
& \les 
T^{\eps} \| \Xi_k \|_{C_T H_x^{-1+\eps}} \leq T^{\eps} \| \pmb{\Xi} \|_{\mathcal{X}^s}.
\end{split}
\label{L4}
\end{equation}

\noi
Putting \eqref{L1}, \eqref{L2}, \eqref{L3}, and \eqref{L4} together, we obtain
\[ \| \vec \G(v) \|_{C_T \H_x^1} \leq C_1 \| (v_0,v_1) \|_{\H^1} 
+ C_2 T^\ta \big(1 + \| \pmb{\Xi} \|_{\mathcal{X}^s}\big) \big(1 + \| \vec v \|_{C_T \H_x^1}\big)^k \]

\noi
for some $\ta > 0$.
A similar computation yields
a difference estimate on $\vec \G(v_1)- \vec \G(v_2)$. Therefore, by choosing $T = T( \| \pmb{\Xi} \|_{\mathcal{X}^s} ) > 0$ sufficiently small, we conclude that $\G$ is a contraction in the ball $B_R \subset C([0,T]; \H^1(\T^2))$ of radius $R\sim  \| (v_0,v_1) \|_{\H^1} $.
\end{proof}

\begin{remark}\label{REM:LWP3} \rm
(i) In view of \eqref{enh1}
and an analogue of Lemma \ref{LEM:Psi} 
(for a rougher noise), we see that 
the proof of Proposition \ref{PROP:LWP2}
yields local well-posedness of \eqref{SNLW9}
with a rougher noise $\sqrt 2 D^\al \xi$
for $\al < \frac 12 + \frac 1k$
such that 
the corresponding enhanced data set $\pmb{\Xi} = (v_0, v_1,  \Psi, :\!\Psi^2\!:, \dots, :\!\Psi^k\!:\, )$
 belongs to 
 $\mathcal{X}^{s}(\T^2)$ almost surely.
It may be possible to improve the local well-posedness argument above
by using the wave Strichartz estimates.

\smallskip

\noi
(ii) 
Suppose that  $(v_0,v_1)$ lies in $\H^s (\T^2)$ for some $ s < 1$. 
From \eqref{S2} and Lemma \ref{LEM:Sch}, we have 
\[ \| V(t)(v_0,v_1) \|_{H_x^1} \les t^{s-1} \| (v_0,v_1) \|_{\H^s} \]
for any $0 < t \leq 1$. 
Then, given $0 < T \leq 1$, we define 
 the $Y^{1, s} (T)$-norm by 
\begin{equation*}
\| (v, \dt v)  \|_{Y^{1, s} (T)} = \sup_{0 \leq t \leq T} t^{1-s} \| (v, \dt v) (t) \|_{\H^1_x}, 
\end{equation*}

\noi
where a function is allowed to blow up at time $t = 0$.
 By slight modifications of \eqref{L2}, \eqref{L3}, and \eqref{L4}, we obtain
\[ \| \vec \G(v) \|_{Y^{1, s} (T)} \les \| (v_0,v_1) \|_{\H^s} +  T^{\ta} 
\big(1 + \| \pmb{\Xi} \|_{\mathcal{X}^s}\big) 
\big(1 + \| \vec v \|_{Y^{1, s} (T)}\big)^k \]
for some $\ta>0$,  provided that $s > \frac{k-1}{k}$. 
A similar computation also yields a difference estimate.
This proves existence of the unique solution 
 $(v, \dt v)  \in C((0,T]; \H^1(\T^2))
 \cap C([0,T]; \H^s(\T^2))$, where the latter regularity may be shown a posteriori. 
\end{remark}

\section{Pathwise global well-posedness in the cubic case}
\label{SEC:GWP1}
In this section, we prove pathwise global well-posedness of 
\eqref{SNLW11}
with 
$(v, \dt v)|_{t = 0}
= (u_0, u_1)
 \in \H^1 (\T^2)$ and $\wick{\Psi^\l} \, \in C([0,T]; W^{-\eps,\infty}(\T^2))$, $\eps> 0$, 
 almost surely
 for 
$\ell = 1, 2, 3$.
As mentioned in Section \ref{SEC:1}, 
our main goal is to control the growth of 
  the energy $E(\vec v)$ in \eqref{E1} (with $k = 3$).

Given $T> 0$,  we fix $0 \leq t \leq T$ and suppress $t$-dependence in the following.
By \eqref{E1} and \eqref{SNLW11}, we have 
\begin{align}
\begin{split}
\dt E (\vec v) &= \int_{\T^2} (\dt v) \big\{ \dt^2 v + (1-\Dl) v + v^3 \big\} dx \\
&= -3\int_{\T^2} (\dt v) v^2 \Psi dx - 3\int_{\T^2} (\dt v) v \wick{\Psi^2} dx - \int_{\T^2} (\dt v) \wick{\Psi^3} dx \\
&\quad - \int_{\T^2} \big( D^{\frac{1}{2}} \dt v \big)^2 dx \\
&=: A_1 + A_2 + A_3 - \| D^{\frac{1}{2}} \dt v \|_{L^2}^2.
\end{split}
\label{E2}
\end{align}

\noi
\noi
Given $T\gg1$,   we set $B( T) = B(T; \Psi)$ by 
\begin{align}
B( T) 
& = 1 
+ \|  \wick{\Psi^2} \|_{L^\infty_TW_x^{-\frac 12, \infty}}^2
+ \|  \wick{\Psi^3} \|_{L^\infty_TW_x^{-\frac 12, \infty}}^2
+ \eps  \| \Psi \|_{L^\infty_TW_x^{-\eps, \infty}}^2
\label{E1a}
\end{align}

\noi
for some small $\eps > 0$.
By Cauchy-Schwarz's
 inequality, Cauchy's inequality, and  Lemma~\ref{LEM:bilin}\,(ii)
 with \eqref{E1a}, we have 
\begin{align}
\begin{split}
|A_2| &\leq 3 \| \jb{\nb}^{\frac{1}{2}} \dt v \|_{L^2} 
\| \jb{\nb}^{-\frac{1}{2}}  v \wick{\Psi^2} \|_{L^2} \\
&\leq C  \| \dt v \|_{L^2}^2 + \dl \| D^{\frac{1}{2}} \dt v \|_{L^2}^2  
+ C_\dl  \| \jb{\nb}^{\frac{1}{2}} v \|_{L^2}^2 \| \jb{\nb}^{-\frac{1}{2}} \wick{\Psi^2} \|_{L^\infty}^2 \\
&\leq C_\dl B(T) E(\vec v) +  \dl \| D^{\frac{1}{2}} \dt v \|_{L^2}^2
\end{split}
\label{E3}
\end{align}

\noi
for some small $\dl > 0$.
Similarly, we have 
\begin{align}
\begin{split}
|A_3| &\leq \| \jb{\nb}^{\frac{1}{2}} \dt v \|_{L^2} \| \jb{\nb}^{-\frac{1}{2}} \wick{\Psi^3} \|_{L^2} \\
&\leq C \| \dt v \|_{L^2}^2 + \dl \| D^{\frac{1}{2}} \dt v \|_{L^2}^2  
+ C_\dl \| \jb{\nb}^{-\frac{1}{2}} \wick{\Psi^3} \|_{L^\infty}^2 \\
&\leq C E(\vec v) + C_\dl B(T) + \dl \| D^{\frac{1}{2}} \dt v \|_{L^2}^2.
\end{split}
\label{E4}
\end{align}

It remains to estimate  $A_1$. 
First,  note that from the embedding $L^{\frac 4{1+\eps}}(\T^2)\subset L^{2+\eps}(\T^2)$ and interpolation, we have 
\begin{equation}
\| v\|_{W^{\eps, 2+\eps}}
\les 
\| v\|_{W^{\eps, \frac 4 {1+\eps}}}
\les \|v\|_{H^1}^\eps \|v\|_{L^4}^{1-\eps}
  \les E^\frac \eps2( \vec v)  E^\frac {1-\eps}4(\vec v)
= E^{\frac{1+\eps}{4}} (\vec v)
\label{E5}
\end{equation}
for sufficiently small $0 < \eps \ll 1$. 
For simplicity of notation, we set $E(t) = E(\vec v(t))$. 
In the following, 
we assume that $E(t) >  1$.

For  $0 < t_1 \leq t_2 \leq T$, let
\begin{equation*}
\mathcal A_1(t_1, t_2)  = \int_{t_1}^{t_2} \int_{\T^2} \dt v \cdot v^2\cdot \Psi \, dx dt.
\end{equation*}

\noi
Then, 
from H\"older's inequality, Lemma \ref{LEM:bilin}\,(i), 
Cauchy's inequality,  Sobolev's inequality, 
and \eqref{E5} with \eqref{E1a}, we have 
\begin{align}
\begin{split}
|\mathcal A(t_1, t_2) | & = \bigg|\int_{t_1}^{t_2} \int_{\T^2} \jb{\nb}^{\eps} \big(\dt v \cdot v^2\big) \cdot\jb{\nb}^{-\eps} \Psi \, dx dt\bigg|\\
& \le C \int_{t_1}^{t_2}\| \dt v (t')\|_{W^{\eps, 2+\eps}} \|v(t') \|_{L^4}^2   dt'
\cdot \| \Psi \|_{L^\infty_T W_x^{-\eps, \infty}}\\
& \quad 
+  C \int_{t_1}^{t_2}
\| \dt  v (t')\|_{L^4}
\|v(t') \|_{W^{\eps, 2+\eps}}
\|v(t') \|_{L^4}   dt'
\cdot \| \Psi \|_{L^\infty_TW_x^{-\eps, \infty}}\\
& \le\int_{t_1}^{t_2} 
\Big(\dl  \| D^\frac 12 \dt v (t')\|_{L^2}^2
+    \eps^{- 1} C_\dl B(T) E(t') \Big) dt'\\
&\quad  +    \eps^{-1} C_\dl B(T)\int_{t_1}^{t_2}
\|v(t') \|_{W^{\eps, 2+\eps}}^2 E(t')^{\frac 12} dt'\\
& \le \dl \int_{t_1}^{t_2} 
 \| D^{\frac 12} \dt v (t')\|_{L^2}^2dt'
+   \eps^{-1}  C_\dl B(T)\int_{t_1}^{t_2}
E^{1 + \frac{\eps}{2}}(t') dt'
\end{split}
\label{E7}
\end{align}

\noi
for any sufficiently small $\eps > 0$.
Let us now consider the second term on the right-hand side of \eqref{E7}
with $p = 2\eps^{-1}$:
\begin{align*}
A(p) = p\int_{t_1}^{t_2}
E^{1+ \frac 1 p}(t') dt'.
\end{align*}

\noi
By optimizing $A(p)$ in $p$  at $p =  \log E(t')$ (for each fixed $t'$)
and noting that $x^\frac{c}{\log x} \les 1$ for any $x > 1$, 
it follows from \eqref{E7} that 
\begin{align}
\begin{split}
|\mathcal A(t_1, t_2) |
& \le \dl \int_{t_1}^{t_2} 
\| D^\frac 12 \dt v (t')\|_{L^2}^2 dt'
+   C_\dl B(T) \int_{t_1}^{t_2}
E(t') \log E(t')dt'.
\end{split}
\label{E8}
\end{align}

\noi
Combining \eqref{E2}, \eqref{E3}, \eqref{E4}, and \eqref{E8}, 
we then obtain the following Gronwall bound:
\begin{align}
E(t_2) - E(t_1) 
& \leq
  C\cdot B(T) \int_{t_1}^{t_2}
E(t') \log E(t')dt'
\label{E9}
\end{align}

\noi
for any $0 < t_1 \leq t_2 \leq T$. By solving \eqref{E9}, we obtain
\begin{equation}
E(t) \les e^{e^{C\cdot B(T)t}}
\label{E10}
\end{equation}

\noi
for any $0 < t \leq T$.

Lastly, note that from Lemma \ref{LEM:Psi}, 
we see that $B(T) < \infty$,  almost surely.
Furthermore, the choice of $T \gg1 $ was arbitrary.
Therefore, we conclude that 
the $\H^1$-norm of the solution $(v, \dt v)(t)$
to \eqref{SNLW11} remains finite
on any finite time interval $[0, T]$, almost surely, 
thus allowing us to iteratively apply Proposition \ref{PROP:LWP2}.
This concludes the proof of Theorem~\ref{THM:GWP1}.

\begin{remark}\label{REM:2}\rm
(i) In order to justify 
the formal computation in this subsection, 
we need to proceed with the smooth solution $(v_N , \dt v_N)$
to the truncated equation \eqref{SNLW6a}
with  the frequency truncated  initial data $(\P_N u_0, \P_N u_1)$
(for example, to guarantee  finiteness
of the last term on the right-hand side of \eqref{E2})
and then take $N \to \infty$ in \eqref{E10} by  noting that the implicit constant is
independent of $N \in \N$ and by using Proposition \ref{PROP:LWP2}
(namely, the continuous dependence of a solution on an enhanced data set).
This argument, however, is standard and thus we omit details.
See, for example,  \cite{OP1}.

\smallskip

\noi
(ii) By refining the argument, it is possible to obtain a double exponential bound
\begin{equation*}
\|(v, \dt v)(t)\|_{\H^1} \leq C e^{e^{C(\o)t^\ta}}
\end{equation*}

\noi
for some $\ta > 0$ and an almost surely finite random constant $C(\o)>0$.
We, however, do not pursue this issue here.
See, for example, 
Remark 3.7 in \cite{GKOT}.

\end{remark}

\section{Almost sure global well-posedness}\label{SEC:GWP2}

In this section, we first sketch the proof of almost sure global well-posedness
with the Gaussian random initial data $(u_0^\o, u_1^\o)$ in \eqref{IV1}
(Theorem \ref{THM:GWP2}), 
and then briefly discuss the proof of 
almost sure global well-posedness
with the shifted initial data $(v_0, v_1) + (u_0^\o, u_1^\o)$
(Corollary \ref{COR:GWP3}).

\subsection{Invariant measure argument}
\label{SUBSEC:GWP2a}

As mentioned in Section \ref{SEC:1}, 
the main strategy for proving Theorem \ref{THM:GWP2} is to 
(i)~use the mutual absolute continuity
of $\muu_1$ (the law of the random initial data 
$(u_0^\o, u_1^\o)$ in \eqref{IV1})
and the Gibbs measure $\rhoo$ in~\eqref{Gibbs1}
and 
(ii)~then apply Bourgain's invariant measure argument.
In the following, 
given a random variable $X$, 
let $\L(X)$ denote the law of $X$.

For this purpose, we first review the construction of the Gibbs measure $\rhoo$.
Given $N \in \N$, 
consider the  truncated Gibbs measure:
\begin{align}\label{GibbsN}
d\rhoo_N(u,\dt u )= Z_N^{-1}R_N(u)d\muu_1(u,\dt u)
\end{align}

\noi
with the  truncated renormalized density:
\begin{align}
R_N(u) = 
\exp\bigg( - \frac{1}{k+1} 
\int_{\T^2} :\!(\P_N u)^{k+1}(x) \!: dx\bigg), 
\label{R1}
\end{align}

\noi
where the Wick power $:\!(\P_N u)^{k+1} \!:$ is defined by  
\begin{align*}
:\!(\P_N u)^{k+1}(x) \!:
 \, \deff  H_{k+1} (\P_N u(x); \al_N)
\end{align*}

\noi
with $\al_N$ as in \eqref{sig2}.
Then, it is known that $\{R_N \}_{N \in \N}$
converges to some $R(u)$  in $L^p(\muu_1)$
for any finite $p \geq 1$
and thus the truncated Gibbs measure $\rhoo_N$ in \eqref{GibbsN}
converges, say in total variation, 
to the renormalized Gibbs measure $\rhoo$ given by
\begin{align}
\begin{split}
d\rhoo(u,\dt u )
& = Z^{-1} R(u)d\muu_1(u, \dt u )\\
& = Z^{-1} \exp\bigg( - \frac{1}{k+1} 
\int_{\T^2} :\! u^{k+1}(x) \!: dx\bigg)d\muu_1(u, \dt u ).
\end{split}
\label{Gibbs3}
\end{align}

\noi
Furthermore, 
the resulting Gibbs measure $\rhoo$ 
and $\muu_1$ are  mutually absolutely continuous.
See \cite{Simon, GJ, DPT1, OTh1} for details.

Next, 
we turn our attention to the well-posedness part.
Let us consider the following  truncated SvNLW:
\begin{align}
\dt^2 u_N    +(1-\Dl)  u_N +  D \dt u_N 
+
\P_N\big(:\!(\P_N u_N)^{k} \!: \big) 
   = \sqrt{2} D^\frac{1}{2}\xi 
\label{xNLW1}
\end{align} 

\noi
and its formal limit: 
\begin{align}
\dt^2 u     +(1-\Dl)  u  + D \dt u 
+
:\!u^{k} \!: \, 
   = \sqrt{2}  D^\frac{1}{2}\xi  .
\label{xNLW2}
\end{align}

\noi
As we see below, 
 the  truncated Gibbs measure $\rhoo_N$ in \eqref{GibbsN}
 is invariant under the truncated dynamics \eqref{xNLW1}.
Then, Bourgain's invariant measure argument \cite{BO94, BO96} yields
the following almost sure global well-posedness.

 \begin{theorem}\label{THM:GWP3}
Let $k \in 2 \N + 1$.
Then, the  renormalized SvNLW~\eqref{xNLW2} is almost surely globally well-posed with respect to the renormalized Gibbs measure~$\rhoo$ in~\eqref{Gibbs3}. Furthermore, the renormalized Gibbs measure $\rhoo$ is invariant under the dynamics.

More precisely, there exists a non-trivial stochastic process $(u,\dt u)\in C(\R_+;\H^{-\eps}(\T^2))$ for any $\eps>0$ such that, given any $T>0$, the solution $(u_N,\dt u_N)$ to 
the  truncated SvNLW~\eqref{xNLW1} with 
$\L((u_N(0), \dt u _N(0)) = \rhoo_N$ converges in probability to 
some stochastic process $(u,\dt u)$ in $C([0,T];\H^{-\eps}(\T^2))$.
Moreover,  we have  $\L (u(t),\dt u(t)) = \rhoo$ for any $t \ge 0$.
\end{theorem}

The proof of Theorem \ref{THM:GWP3}
follows exactly the same steps as in the proof
of the almost sure global well-posedness to 
the stochastic damped NLW studied in \cite{GKOT}:
\begin{align*}
\dt^2 u + \dt u + (1 -  \Dl)  u   \, + \wick{u^k}\,  = \sqrt 2\xi
\end{align*}

\noi
and thus we only sketch the key steps in the following.

 Let  us first  describe the precise meaning of the renormalizations in \eqref{xNLW1}
 and \eqref{xNLW2}.
 Write the solution 
 $u_N$  to \eqref{xNLW1}
 with $\L\big((u_N, \dt u_N)|_{t=0}\big) = \muu_1$
 as 
\begin{align*}
u_N = v_N + \Phi
= (v_N +  \Phi_N) + \P_N^\perp \Phi,
\end{align*}

\noi
where $\P_N^\perp = \Id - \P_N$
and $\Phi$ denotes the stochastic convolution defined in \eqref{phi2}
(recall that\footnote{A tedious but direct computation, as in \eqref{sig2},  shows that 
$\E\big[|\ft {\dt \Phi}(t, n)|^2 \big] = 1$ for any $n \in \Z^2$.} $\L( (\Phi(t), \dt \Phi(t) ) ) = \muu_1$ for any $t \geq 0$).
Then, 
we see that \eqref{xNLW1}
decouples into 
the linear dynamics for the high frequency part
$\P_N^\perp u_N = \P_N^\perp\Phi$:
\begin{align}
\dt^2 \P_N^\perp\Phi  +(1-\Dl) \P_N^\perp\Phi 
+ D\dt  \P_N^\perp\Phi
 = \sqrt{2} \P_N^\perp
D^\frac{1}{2}\xi
\label{xNLW3}
\end{align}

\noi
and 
the nonlinear dynamics for the low frequency part $\P_N u_N$:
\begin{align}
\dt^2 \P_N u_N     +(1-\Dl)  \P_N u_N 
+ D\dt \P_N u_N
+
\P_N\big(:\!(\P_N u_N)^{k} \!: \big) 
   = \sqrt{2} \P_ND^\frac{1}{2} \xi   .
\label{SNLW11a}
\end{align}

\noi
In terms of  $v_N = \P_N u_N - \Phi_N $, we can write \eqref{SNLW11a}
as 
 \begin{align}
\begin{cases}
\dt^2 v_N + (1-\Dl)v_N  + D\dt v_N +
 \sum_{\ell=0}^k {k\choose \ell} \P_N \big(:\!\Phi_N^\l\!:  v_N^{k-\ell}\big)
=0\\
(v_N,\dt v_N)|_{t = 0}=(0,0), 
\end{cases}
\label{SNLW12}
\end{align}

\noi
where the Wick power is defined by 
\begin{align*}
:\!\Phi_N^\l (t, x) \!:
 \, \deff  H_{\l} (\Phi_N(t, x); \al_N), 
\end{align*}

\noi
where  $\al_N$ is as in \eqref{sig2}.
By taking $N \to \infty$,  
we  obtain the limiting equation:
 \begin{align}
\begin{cases}
\dt^2 v + (1-\Dl)v  +
D\dt v +
 \sum_{\ell=0}^k {k\choose \ell} :\!\Phi^\l\!:  v^{k-\ell}
=0\\
(v,\dt v)|_{t = 0}=(0,0).
\end{cases}
\label{SNLW13}
\end{align}

\noi
In view of Lemma \ref{LEM:Psi}, 
the local well-posedness result (Proposition \ref{PROP:LWP2})
applies to \eqref{SNLW12}
and \eqref{SNLW13}, uniformly in $N \in \N$.
Furthermore, the solution $v_N$
converges to $v$ 
on the (random) time interval of local existence.

Once 
we check invariance of 
the truncated Gibbs measure $\rhoo_N$ in \eqref{GibbsN} under 
the truncated SvNLW  dynamics~\eqref{xNLW1}, 
the rest of the proof of Theorem \ref{THM:GWP3}
follows from a standard application of Bourgain's invariant measure argument.
See, for example,~\cite{ORTz,  OOTol, Bring2, OOTol2}
for details in the context of stochastic nonlinear wave equations.
See also  \cite{OTh2}.

\medskip
Invariance of the truncated Gibbs measure $\rhoo_N$
under 
the truncated SvNLW  dynamics~\eqref{xNLW1} follows
from exactly the same argument presented in Section 4 of \cite{GKOT}.
For readers' convenience, however, 
we sketch the argument.
Given $N \in \N$, define
the marginal probabilities measures
$\muu_{1, N}$ and $\muu_{1, N}^\perp$
on $\P_N \H^{-\eps}(\T^2)$
and $\P_N^\perp \H^{-\eps}(\T^2)$, respectively, 
as 
   the induced probability measures
under the map $T_N$ for $\muu_{1, N}$
and $T_N^\perp$ for $\muu_{1, N}^\perp$,  where 
\begin{equation*}
T_N: \o \in \O \longmapsto (\P_N u_0^\o, \P_N u_1^\o)
\qquad \text{and}
\qquad 
T_N^\perp: \o \in \O \longmapsto (\P_N^\perp u_0^\o, \P_N^\perp u_2^\o), 
 \end{equation*}

\noi
where $(u_0^\o, u_1^\o)$ is as in \eqref{IV1}.
Then, with $\muu_1 = 
\muu_{1, N} \otimes \muu_{1, N}^\perp$, 
it follows from \eqref{GibbsN}  that 
\begin{align}
\rhoo_N = \vec \nu_{N} \otimes \muu_{1, N}^\perp, 
\label{X8a}
\end{align}

\noi
where $\vec \nu_N$ is given by
$d \vec \nu_N= Z_N^{-1}R_N(u)d\muu_{1, N}$
with the density $R_N$ as in \eqref{R1}.
The high frequency dynamics 
\eqref{xNLW3} is linear and thus we can readily verify
that the Gaussian measure $\muu_{1, N}^\perp$
is invariant under the dynamics of \eqref{xNLW3}.
Hence, it remains to check invariance of $\vec \nu_N$
under the low frequency dynamics \eqref{SNLW11a}.

With $(u_N^1, u_N^2) = (\P_N u_N, \dt \P_N u_N)$, we
can write
\eqref{SNLW11a}  in the following Ito formulation:
\begin{align}
\begin{split}
d  \begin{pmatrix}
u_N^1 \\ u_N^2
\end{pmatrix}
& = -  \Bigg\{
\begin{pmatrix}
0  & -1\\
1-\Dl &  0
\end{pmatrix}
 \begin{pmatrix}
u_N^1 \\ u_N^2
\end{pmatrix}
 +  
\begin{pmatrix}
0 \\ \P_N\big( :\!(u_N^1)^k\!:\big)
\end{pmatrix}
\Bigg\} dt \\
& \quad + 
  \begin{pmatrix}
0  \\ - D u_N^2 dt + \sqrt 2\P_N D^\frac{1}{2} dW
\end{pmatrix} .
\end{split}
\label{SNLW16}
\end{align}

\noi
This shows that the generator $\L^N$ of the Markov semigroup for  \eqref{SNLW16}
 can be written as $\L^N = \L^N_1 + \L^N_2$, 
 where $\L^N_1$ denotes the generator
 for the deterministic NLW with the truncated nonlinearity:
\begin{align}
\begin{split}
d  \begin{pmatrix}
u_N^1 \\ u_N^2
\end{pmatrix}
 = -  \Bigg\{
\begin{pmatrix}
0  & -1\\
1-\Dl &  0
\end{pmatrix}
 \begin{pmatrix}
u_N^1 \\ u_N^2
\end{pmatrix}
 +  
\begin{pmatrix}
0 \\ \P_N\big( :\!(u_N^1)^k\!:\big)
\end{pmatrix}
\Bigg\} dt 
\end{split}
\label{SNLW17}
\end{align}

\noi
and $\L^N_2$ denotes the generator
for the Ornstein-Uhlenbeck process: 
\begin{align}
d   u_N^2
   = 
 - D u_N^2 dt + \sqrt 2\P_N  D^\frac{1}{2}dW.
\label{SNLW18}
\end{align}

Invariance of 
 $\vec \nu_N$  under the dynamics of \eqref{SNLW17}
 follows easily 
 from  the Hamiltonian structure of 
  \eqref{SNLW17}  with the Hamiltonian:
\begin{align*}
E_N (u^1_N, u^2_N ) = \frac{1}{2}\int_{\T^2}\big( (u^1_N)^2 +  |\nb u^1_N|^2\big) dx
+ 
\frac{1}{2}\int_{\T^2} (u^2_N)^2dx
- \log \big(R_N(u^1_N)\big), 
\end{align*}
 
\noi
where $R_N$ is as in \eqref{R1}, 
in particular, 
   the conservation of the Hamiltonian 
$E_N(u^1_N, u^2_N )$ and Liouville's theorem
(on a finite-dimensional phase space $\P_N \H^{-\eps}(\T^2)$).
Hence, we have 
 $(\L^N_1)^*\vec \nu_N = 0$.

As for \eqref{SNLW18}, 
recalling that 
 the Ornstein-Uhlenbeck process
preserves the standard Gaussian measure
(at each frequency on the Fourier side), 
we see  that $\vec \nu_N$ is also invariant under the dynamics of \eqref{SNLW18}, 
since, on the second component $u_N^2$, 
 the measure $\vec \nu_N$ is nothing but the  white noise $\mu_0$
(projected onto the low frequencies $\{|n|\leq N\}$).
Hence, we have  $(\L^N_2)^*\vec \nu_N = 0$.
Therefore, we conclude that 
\[(\L^N)^*\vec \nu_N = 
(\L^N_1)^*\vec \nu_N +  (\L^N_2)^*\vec  \nu_N = 0.\]

\noi
This shows invariance of $\vec \nu_N$ under \eqref{SNLW16}
and hence under \eqref{SNLW11a}.
Finally, 
invariance of 
 the truncated Gibbs measure 
$\rhoo_N$ in~\eqref{GibbsN}
under the  truncated SvNLW dynamics \eqref{xNLW1}
follows from 
\eqref{X8a}
and the invariance of   $\muu^\perp_{1, N}$  and $\vec \nu_N$
under \eqref{xNLW3} and \eqref{SNLW16}, respectively.

\begin{remark}\rm
As a consequence of 
 Bourgain's invariant measure argument, 
the solution $(u, \dt u)$ to \eqref{SNLW9}  constructed in Theorem \ref{THM:GWP2}
satisfies 
 the following logarithmic growth bound:
\[ \| (u(t), \dt u(t)) \|_{\H^{-\eps}} \leq C(\o) \big(\log (1 + t)\big)^\frac{k}{2}\]

\noi
for any $t \geq 0$.
See \cite{ORTz} for details.

\end{remark}

\subsection{Almost sure global well-posedness with the shifted initial data}
\label{SUBSEC:GWP2b}

We conclude this paper by briefly discussing the proof of 
Corollary \ref{COR:GWP3}, using the Cameron-Martin theorem \cite{CM}.
For this purpose, we first  go over the definition
of abstract Wiener spaces introduced by
Gross \cite{GROSS}.
See also
Kuo \cite{KUO}.
Let $H$ be a real separable Hilbert space.
It is known that the Gauss measure
$\mu$ with the density
$d \mu = Z^{-1} e^{-\frac{1}{2}\|x\|_{H}^2}dx$
is only finitely additive if $\dim H = \infty$.

Let $\mathcal{P}$ denotes the collection
of all finite dimensional orthogonal projections of $H$.
A seminorm $|||\cdot|||$
on $H$ is said to be measurable
if, for any $\eps > 0$, there exists $\P_\eps \in \mathcal{P}$
such that
$\mu (|||\P x|||>\eps) < \eps$
for all $\P\in \mathcal P$ with $\P\perp \P_\eps$.
Let $B$ be the completion of $H$ with respect to this seminorm $|||\cdot|||$.
Then,
Gross \cite{GROSS} showed that
we can construct 
$\mu$  as  a countably additive
Gaussian measure on $B$.
In this case, we say that
the triplet $(B, H, \mu)$ is an {\it abstract Wiener space}.
The original Hilbert space $H$
is  referred to as a  {\it Cameron-Martin space}
or a reproducing kernel Hilbert space.

Let $(B, H, \mu)$ be an abstract Wiener space.
Then,  the Cameron-Martin Theorem \cite{CM}
states the following.

\noi
\begin{lemma}\label{LEM:CM}
 Given $h \in B$, define
a shifted measure $\mu_h$ by $ \mu_h(\, \cdot\,) := \mu(\, \cdot\,  - h)$.
Then,
the shifted measure $\mu_h$
and the original Gaussian measure
are equivalent (namely, mutually absolutely continuous)
if and only if
$h \in H$.
\end{lemma}

Let $(u_0^\o, u_1^\o)$ be as in \eqref{IV1}.
Then, its distribution 
is given by  $\muu_1 = \mu_1\otimes \mu_0$
defined in~\eqref{gauss0}, with the formal density:
\begin{align*}
 d \muu_1
   = Z^{-1} e^{-\frac 12 \| (u, v)\|_{\H^{1}}^2} d(u, v).
\end{align*}

\noi
In this context, the Cameron-Martin theorem states that 
the Gaussian measure $\muu_1$ 
and the shifted measure 
\begin{align}
\muu_{1, \vec v_0}(\,\cdot\,) : =  \muu_1(\, \cdot\,  - \vec v_0)
\label{mu2}
\end{align}
are equivalent
if and only if $\vec v_0 = (v_0, v_1) $ belongs to the Cameron-Martin space $\H^1(\T^2)$.

Now, fix $\vec v_0  = (v_0, v_1) \in \H^1(\T^2)$.
Let $\Phi[\vec v_0 + \vec u_0^\o] $ be the stochastic convolution
defined in~\eqref{phi3}.
With $\Phi_N[\vec v_0 + \vec u_0^\o]
= \P_N \Phi[\vec v_0 + \vec u_0^\o]$, we 
define the Wick power
$:\! (\Phi_N[\vec v_0 + \vec u_0^\o])^\l  \!:$ as in \eqref{Wick4}.
Then, from Sobolev's embedding (with large $r \gg 1$ such that $\eps r > 4$), 
\eqref{phi3}, \eqref{Wick4}, and Lemma \ref{LEM:bilin}\,(ii)
followed by 
Lemma \ref{LEM:bilin}\,(i)
and Sobolev's inequality, we have
\begin{align}
\begin{split}
\| :\! (\Phi_N & [\vec v_0 +  \vec u_0^\o])^\l  \!:\|_{C_T W^{- \eps, \infty}_x}
 \les \| :\! (V(t)\vec v_0 + \Phi_N[\vec u_0^\o])^\l  \!:\|_{C_T W^{- \frac \eps 2, r}_x}\\
&  \les  \sum_{j = 0}^\l
\|(V(t)\vec v_0)^j\|_{C_T W^{\frac \eps 2, p}_x}  \|:\!(\Phi_N[u_0^\o])^{\l-j}  \!:\|_{C_T W^{-\frac \eps 2, \infty}_x}\\
& \les  \sum_{j = 0}^\l
\|(V(t)\vec v_0)\|_{C_T W^{\frac \eps 2, \frac{4}{\eps}}_x}
\|(V(t)\vec v_0)\|_{C_T L^{\frac{4p(j-1)}{4 - \eps p}}_x}^{j-1}  \|:\!(\Phi_N[u_0^\o])^{\l-j}  \!:\|_{C_T W^{-\frac \eps 2, \infty}_x}\\
& \les  \sum_{j = 0}^\l
\|(V(t)\vec v_0)\|_{C_T H^1_x}^j
  \|:\!(\Phi_N[u_0^\o])^{\l-j}  \!:\|_{C_T W^{-\frac \eps 2, \infty}_x}\\
 & < \infty,  
\end{split}
\label{XX}
\end{align}

\noi
almost surely, 
thanks to the fact that $\vec v_0 \in \H^1(\T^2)$ and Lemma \ref{LEM:Psi}.
Here,  we chose $p \ge 1$ such that $\frac \eps 4 < \frac 1 p < \frac 1 r + \frac \eps 4$.
We also point out that the implicit constants in \eqref{XX} are independent of 
the cutoff size $N \in \N$.
A slight modification of \eqref{XX}, combined with Lemma \ref{LEM:Psi}, 
allows us to construct
the Wick power
$:\! (\Phi [\vec v_0 +  \vec u_0^\o])^\l  \!:$ by a limiting procedure.

Given $\vec w_0 = (w_0, w_1) \in \H^{-\eps}(\T^2)$, 
consider the following Cauchy problem:
\begin{equation}
\begin{cases}
\dt^2 v + (1-\Dl) v + D \dt v + \sum_{\ell = 0}^k \binom{k}{\ell} 
\wick{(\Phi
[\vec w_0])^\ell} v^{k-\ell} = 0\\
(v, \dt v)|_{t = 0} = (0, 0).
\end{cases}
\label{SNLWx}
\end{equation}

\noi
In the following, we write the solution $\vec v = (v , \dt v)$ to \eqref{SNLWx}
as $\vec v[\vec w_0]$ to signify the dependence on $\vec w_0$.
Here, the Wick power $\wick{(\Phi
[\vec w_0])^\ell}$, if it exists, is defined by a limiting procedure as above
with the divergent constant $\al_N$ in \eqref{sig2}.
For $\vec w_0 = \vec u_0^\o$ distributed by $\muu_1$, 
the Wick power exists almost surely.
Furthermore, 
 as discussed in 
Subsection \ref{SUBSEC:GWP2a},
by Bourgain's invariant measure argument, 
  the solution $v[\vec u_0^\o]$
to~\eqref{SNLWx} exists globally in time, almost surely.
In particular, by a close inspection of the argument (see, for example, \cite{ORTz}), we can also show that 
\begin{align}
\sup_{t \geq 0 } \frac{\| \vec v[\vec u_0^\o](t) \|_{\H^{1-\eps}}}{C(t)} < \infty, 
\label{XX1}
\end{align}

\noi
almost surely, for some deterministic increasing function $C(t)$.\footnote{A polynomial growth would suffice.}

We now define a set $\mathcal{A} \subset \H^{-\eps}(\T^2)$ by 
\begin{align}
\mathcal{A} = \Bigg\{ \vec w_0 \in \H^{-\eps}(\T^2): 
\sup_{t \geq 0 } \frac{\| \vec v[\vec w_0](t) \|_{\H^{1-\eps}}}{C(t)} < \infty\bigg\}.
\label{XX2}
\end{align}

\noi
From the construction, we see that  the map $\vec w_0 \mapsto
\{\wick{(\Phi
[\vec w_0])^\ell}\}_{\l = 1}^k$ is  measurable
(but is not continuous).
Furthermore, 
the solution map 
$( \Phi[\vec w_0], \wick{(\Phi[\vec w_0])^2}, \dots, 
\wick{(\Phi[\vec w_0])^k}) 
\mapsto \vec v[\vec w_0]$
is continuous; see, for example, Proposition \ref{PROP:LWP2}.
Hence, the map $w_0 \mapsto  \vec v[\vec w_0]$ is measurable
and thus the set $\mathcal A$ in \eqref{XX2} is a measurable set.
Therefore, from 
Lemma \ref{LEM:CM} with \eqref{XX1}
(namely $\muu_1 (\mathcal A) = 1$), 
we conclude that 
$\muu_{1, \vec v_0}(\mathcal{A}) = 1$,
where $\muu_{1, \vec v_0}$ is as in 
\eqref{mu2}.
Noting that 
$\muu_{1, \vec v_0}$ is the distribution
of the shifted initial data $\vec v_0 + \vec u_0^\o$, 
we conclude that the $\H^{1-\eps}$-norm
of the solution $\vec v $ to~\eqref{SNLW7b}
remains finite for finite times, almost surely.
This proves Corollary \ref{COR:GWP3}.

%
%
%

\begin{ackno}\rm
R.L.~and T.O.~were supported by the European Research Council (grant no.~864138 ``SingStochDispDyn'').
The authors would like to thank Leonardo Tolomeo 
for a helpful discussion.
They are also grateful to the anonymous referees for their  comments.

\end{ackno}

\end{document}